\documentclass[11pt,a4paper]{amsart}
\usepackage[utf8]{inputenc}
\usepackage{amsmath,amssymb, amsthm, amsfonts, gensymb, commath, float, mathtools, enumerate, breqn}
\usepackage{tabularx}
\usepackage{multirow}
\usepackage{esvect}
\usepackage{subcaption}
\usepackage{tikz}
\usetikzlibrary{positioning, arrows, arrows.meta, cd}
\usetikzlibrary{shapes,backgrounds,shapes,positioning,petri,topaths}

\theoremstyle{theorem}
\newtheorem{theorem}{Theorem}[section]
\newtheorem{lemma}[theorem]{Lemma}
\newtheorem{conjecture}[theorem]{Conjecture}
\newtheorem{corollary}[theorem]{Corollary}
\newtheorem{proposition}[theorem]{Proposition}

\def\beq#1#2\eeq{%
        \begin{equation}%
        \label{#1}%
            #2%
        \end{equation}%
   }

\newcommand{\Ord}{\mathrm{O}}

\usepackage{color}
\usepackage{graphicx}
\usepackage{epstopdf}
\usepackage[text={15cm,23cm}]{geometry} 
\usepackage[colorlinks=true,%
            linkcolor=red!50!black,%
            citecolor=blue!50!black,%
            urlcolor=darkgray]{hyperref}  

\DeclareMathOperator{\Conv}{Conv}

\theoremstyle{definition}
\newtheorem{example}[theorem]{Example}
\newtheorem{definition}[theorem]{Definition}

\DeclareMathOperator{\Hess}{Hess}

\title[Markov Polynomials]{Arithmetic and Geometry of Markov polynomials}

\author{S.J. Evans}
\address{Department of Mathematical Sciences,
Loughborough University, Loughborough LE11 3TU, UK}
\email{S.J.Evans@lboro.ac.uk}

\author{A.P. Veselov}
\address{Department of Mathematical Sciences,
Loughborough University, Loughborough LE11 3TU, UK}
\email{A.P.Veselov@lboro.ac.uk}

\author{B. Winn}
\address{Department of Mathematical Sciences,
Loughborough University, Loughborough LE11 3TU, UK}
\email{B.Winn@lboro.ac.uk}

\begin{document}

\maketitle

\begin{abstract}
    Markov polynomials are the Laurent-polynomial solutions of the generalised Markov equation $$X^2 + Y^2 + Z^2 = kXYZ, \quad k=\frac{x^2 + y^2 + z^2}{x y z}$$ which are the results of cluster mutations applied to the initial triple $(x, y, z)$. They were first introduced and studied by Itsara, Musiker, Propp and Viana, who proved, in particular, that their coefficients are non-negative integers. We study the coefficients of Markov polynomials as functions on the corresponding Newton polygons, proposing several new conjectures.
Some of these conjectures are proved for the special cases of Markov polynomials corresponding to Fibonacci and Pell numbers.   
  \end{abstract}

\section{Introduction}

One of the most remarkable Diophantine equations is the celebrated Markov 
equation
$$X^2 + Y^2 + Z^2=3XYZ.$$ Markov showed \cite{Markov} that all its solutions 
in positive integers can be found recursively from $(1,1,1)$ using the Vieta 
involution 
\beq{Vieta}
(X,Y,Z)\mapsto \left(X,Y, Z'\right),\qquad Z' = \frac{X^2+Y^2}{Z},
\eeq
and permutation of variables.

Following  Itsara, Musiker, Propp and Viana \cite{IMPV,Pro}, who were inspired by the earlier work of Fomin and Zelevinsky \cite{FZ}, we consider the following generalisation of 
the Markov equation 
\beq{MEq}
    X^2 + Y^2 + Z^2 = k(x, y, z) XYZ, \quad k(x, y, z)= \frac{x^2 + y^2 + z^2}{x y z}.
\eeq
with the solutions being rational functions of the parameters $x,y,z$. Starting with the initial solution $(X,Y,Z)=(x,y,z)$ we can apply the Vieta involutions \eqref{Vieta} to arrive at solutions $$X=X(x,y,z),\; Y=Y(x,y,z),\; Z=Z(x,y,z)$$ of the equation (\ref{MEq}), certain Laurent polynomials of $x,y,z$, which
we call {\it Markov polynomials.} 

The Laurent property follows from general cluster algebra framework of Fomin and Zelevinsky \cite{FZ}, but in this case also from the alternative form of the Vieta involution 
\begin{equation}
\label{eq:vieta2}
Z'=k(x,y,z)XY-Z.
\end{equation}
In \cite{Pro} it was shown that the coefficients of Markov polynomials are non-negative. When we specialise all the parameters $x,y,z$ to $1$, we get the celebrated Markov numbers, playing seminal role in many areas of mathematics \cite{Aig}. In this sense Markov polynomials can be viewed as a 3-parameter quantisation of Markov numbers.

The aim of this paper is to study these polynomials in more detail. We will view their coefficients as the functions on the corresponding Newton polygons, which we will describe explicitly.
Using the geometry of the Newton polygons, we prove some new results and state several conjectures about the coefficients of Markov polynomials.
Some of these conjectures are proved for the special series of Markov polynomials corresponding to Fibonacci and Pell numbers.  
We discuss also the continuum limit, introducing the corresponding entropy function, which is conjectured to be concave.  

\section{Markov polynomials on the Conway Topograph}

To represent Markov numbers it is now customary (see e.g. \cite{SV}) to use the Conway topograph introduced by J.H. Conway \cite{Con} to vizualise the values of quadratic forms. 

The Conway topograph consists of the planar domains which are connected components of the complement to the trivalent  tree imbedded 
in the plane. These domains were originally labelled by superbases in the integer lattice $\mathbb Z^2$, but can be also labelled by the rational numbers using the Farey mediant $\frac{p_1}{q_1}\oplus \frac{p_2}{q_2}=\frac{p_1+p_2}{q_1+q_2}$.
We will only be interested in the rationals in $[0, 1]$, leading to the rational Conway topograph given in Figure \ref{fig:Rational Conway}. Neighbouring regions are occupied by Farey neighbours, i.e. $\frac{a}{b}, \frac{c}{d}$ where $\abs{ad-bc} = 1$. 

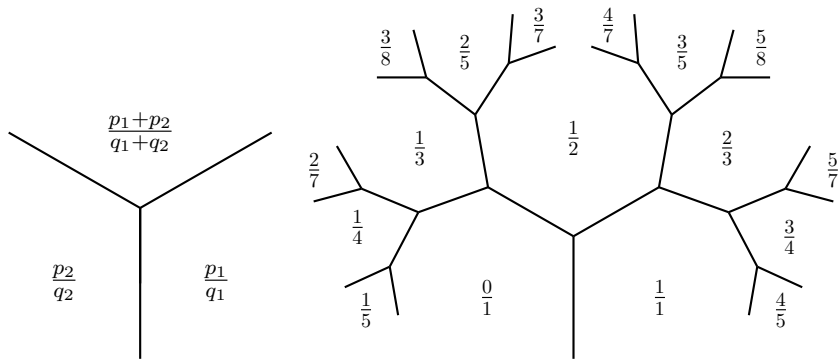
\begin{figure}[H]
\begin{center}
\begin{tabular}{c c}
	\begin{tikzpicture} 
	
	\node at (1.0,1) {$\frac{p_1}{q_1}$};
	\node at (-1.0,1) {$\frac{p_2}{q_2}$};
	\node at (0,3) {$\frac{p_1+p_2}{q_1+q_2}$};
	

        \draw[thick] (0, 0) -- (0, 2);
        \draw[thick] (0, 1) -- (0, 2);
	\draw[thick] (0,2) -- (1.73,3);
	\draw[thick] (0,2) -- (-1.73,3);

	\end{tikzpicture}

&
\begin{tikzpicture}[scale=0.65, every node/.style={scale=0.75}]
    \node at (1.75,0.75) {\Large$\frac{1}{1}$};
    \node at (-1.75,0.75){\Large$\frac{0}{1}$};
    \node at (0,4) {\Large$\frac{1}{2}$};
    \node at (3.11, 3.8) {\Large$\frac{2}{3}$};
    \node at (2.20, 5.76) {\Large$\frac{3}{5}$}; 
    \node at (4.37, 1.98) {\Large$\frac{3}{4}$}; 
    \node at (0.7, 6.3) {\Large$\frac{4}{7}$}; 
    \node at (3.8, 5.85) {\Large$\frac{5}{8}$}; 
    \node at (5.25, 3.35) {\Large$\frac{5}{7}$}; 
    \node at (4.2, 0.5) {\Large$\frac{4}{5}$}; 
    \node at (-3.11, 3.8) {\Large$\frac{1}{3}$};
    \node at (-2.20, 5.76) {\Large$\frac{2}{5}$}; 
    \node at (-4.37, 2.18) {\Large$\frac{1}{4}$}; 
    \node at (-0.7, 6.3) {\Large$\frac{3}{7}$}; 
    \node at (-3.8, 5.85) {\Large$\frac{3}{8}$}; 
    \node at (-5.25, 3.35) {\Large$\frac{2}{7}$}; 
    \node at (-4.2, 0.5) {\Large$\frac{1}{5}$}; 
    
        
        \draw[thick] (0, -0.5)  -- (0, 2);
    
        \draw[thick] (0,2) -- (1.73,3);
            
            \draw[thick] (1.73, 3) -- (1.99, 4.48);	
    
                \draw[thick] (1.99, 4.48) -- (1.31, 5.53);
                
                    \draw[thick] (1.31, 5.53) -- (0.36, 5.87);
                    \draw[thick] (1.31, 5.53) --  (1.23, 6.53);
                    
                \draw[thick] (1.99, 4.48) -- (2.98, 5.24);	
                
                    \draw[thick] (2.98, 5.24) -- (3.24, 6.21);
                    \draw[thick] (2.98, 5.24) -- (3.98, 5.24);
    
            \draw[thick] (1.73, 3) -- (3.14, 2.49);
    
                \draw[thick] (3.14, 2.49) -- (4.29, 2.97);
    
                    \draw[thick] (4.29, 2.97) -- (4.79, 3.84);
                    \draw[thick] (4.29, 2.97) -- (5.26, 2.71);
                
                \draw[thick] (3.14, 2.49) -- (3.72, 1.38);
                    
                    \draw[thick] (3.72, 1.38) -- (4.63, 0.96);
                    \draw[thick] (3.72, 1.38) -- (3.55, 0.39);
                    
        \draw[thick] (0,2) -- (-1.73,3);
        
            \draw[thick] (-1.73, 3) -- (-1.99, 4.48);	
        
                \draw[thick] (-1.99, 4.48) -- (-1.31, 5.53);
        
                    \draw[thick] (-1.31, 5.53) -- (-0.36, 5.87);
                    \draw[thick] (-1.31, 5.53) -- (-1.23, 6.53);
        
                \draw[thick] (-1.99, 4.48) -- (-2.98, 5.24);	
        
                    \draw[thick] (-2.98, 5.24) -- (-3.24, 6.21);
                    \draw[thick] (-2.98, 5.24) -- (-3.98, 5.24);
        
            \draw[thick] (-1.73, 3) -- (-3.14, 2.49);
        
                \draw[thick] (-3.14, 2.49) -- (-4.29, 2.97);
        
                    \draw[thick] (-4.29, 2.97) -- (-4.79, 3.84);
                    \draw[thick] (-4.29, 2.97) -- (-5.26, 2.71);
        
                \draw[thick] (-3.14, 2.49) -- (-3.72, 1.38);
        
                    \draw[thick] (-3.72, 1.38) -- (-4.63, 0.96);
                    \draw[thick] (-3.72, 1.38) -- (-3.55, 0.39);
\end{tikzpicture}
\end{tabular}
\end{center}
\caption{Rationals in $[0, 1]$ represented on the Conway topograph.} \label{fig:Rational Conway}
\end{figure}

To represent Markov numbers on the Conway topograph we iterate \textit{via}
the Vieta formula from the Markov equation, as shown on the left of Figure \ref{fig:Markov Conway}, starting with the initial state $X = Y = Z = 1$. Markov triples can then be seen by looking at the three domains adjacent to a vertex. The first few iterations of this process can also be seen in Figure \ref{fig:Markov Conway}.

Juxtaposition with the rational Conway topograph establishes the parametrization of the Markov numbers by rational numbers $\rho \in [0,1]$, a way of
parametrizing Markov numbers which goes back to Frobenius \cite{Frob}.

\begin{figure}[H]
\begin{center}
\begin{tabular}{c c}
    \begin{tikzpicture}
        \node at (0, 0) {$Z$};
        \node at (-0.75, 1.75) {$X$};
        \node at (0.75, 1.75) {$Y$};
        \node at (0, 4) {$Z' = \frac{X^2 + Y^2}{Z}$};
 
        \draw[thick] (0, 1) -- (0, 2.5);
        \draw[thick] (-1, 0) -- (0, 1);
      
        \draw[thick] (0, 1) -- (1, 0);
        \draw[thick] (-1, 3.5) -- (0, 2.5);
        \draw[thick] (0, 2.5) -- (1, 3.5);
    \end{tikzpicture}
&
    \begin{tikzpicture}[scale=0.65, every node/.style={scale=0.75}]
        \node at (1.75,0.75) {$2$};
	\node at (-1.75,0.75){$1$};
	\node at (0,4) {$5$};
	\node at (3.11, 3.8) {$29$};
	\node at (2.20, 5.76) {$433$}; 
	\node at (4.37, 1.98) {$169$}; 
	\node at (0.7, 6.3) {$6466$}; 
	\node at (3.8, 5.85) {$37666$}; 
	\node at (5.25, 3.35) {$14701$}; 
	\node at (4.2, 0.5) {$985$}; 
	\node at (-3.11, 3.8) {$13$};
	\node at (-2.20, 5.76) {$194$}; 
	\node at (-4.37, 1.98) {$34$}; 
	\node at (-0.7, 6.3) {$2897$}; 
	\node at (-3.8, 5.85) {$7561$}; 
	\node at (-5.25, 3.35) {$1325$}; 
	\node at (-4.2, 0.5) {$89$}; 
	
        \draw[thick] (0, -0.5) -- (0, 2);
	
		\draw[thick] (0,2) -- (1.73,3);
			
			\draw[thick] (1.73, 3) -- (1.99, 4.48);	
	
				\draw[thick] (1.99, 4.48) -- (1.31, 5.53);
				
					\draw[thick] (1.31, 5.53) -- (0.36, 5.87);
					\draw[thick] (1.31, 5.53) --  (1.23, 6.53);
					
				\draw[thick] (1.99, 4.48) -- (2.98, 5.24);	
				
					\draw[thick] (2.98, 5.24) -- (3.24, 6.21);
					\draw[thick] (2.98, 5.24) -- (3.98, 5.24);
	
			\draw[thick] (1.73, 3) -- (3.14, 2.49);
	
				\draw[thick] (3.14, 2.49) -- (4.29, 2.97);
	
					\draw[thick] (4.29, 2.97) -- (4.79, 3.84);
					\draw[thick] (4.29, 2.97) -- (5.26, 2.71);
				
				\draw[thick] (3.14, 2.49) -- (3.72, 1.38);
					
					\draw[thick] (3.72, 1.38) -- (4.63, 0.96);
					\draw[thick] (3.72, 1.38) -- (3.55, 0.39);
					
		\draw[thick] (0,2) -- (-1.73,3);
		
			\draw[thick] (-1.73, 3) -- (-1.99, 4.48);	
		
				\draw[thick] (-1.99, 4.48) -- (-1.31, 5.53);
		
					\draw[thick] (-1.31, 5.53) -- (-0.36, 5.87);
					\draw[thick] (-1.31, 5.53) -- (-1.23, 6.53);
		
				\draw[thick] (-1.99, 4.48) -- (-2.98, 5.24);	
		
					\draw[thick] (-2.98, 5.24) -- (-3.24, 6.21);
					\draw[thick] (-2.98, 5.24) -- (-3.98, 5.24);
		
			\draw[thick] (-1.73, 3) -- (-3.14, 2.49);
		
				\draw[thick] (-3.14, 2.49) -- (-4.29, 2.97);
		
					\draw[thick] (-4.29, 2.97) -- (-4.79, 3.84);
					\draw[thick] (-4.29, 2.97) -- (-5.26, 2.71);
		
				\draw[thick] (-3.14, 2.49) -- (-3.72, 1.38);
		
					\draw[thick] (-3.72, 1.38) -- (-4.63, 0.96);
					\draw[thick] (-3.72, 1.38) -- (-3.55, 0.39);
	\end{tikzpicture}
    \end{tabular}
\end{center}
\caption{Markov numbers represented on the Conway topograph.} \label{fig:Markov Conway}
\end{figure}
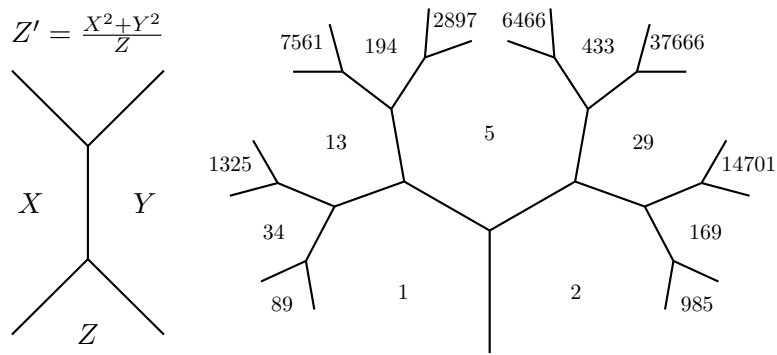


If instead of $(1,1,1)$ on the Conway topograph we start with the triple $(x,y,z)$ and apply the same Vieta formula we get Laurent polynomial solutions of the equation (\ref{MEq}).

\begin{figure}[H]
\begin{center}
	\begin{tikzpicture}
    \node at (1.75,0.75) {$M_{1/1} = \frac{x^2 + y^2}{z}$};
	\node at (-1.75,0.75){$M_{0/1} = x$};
        \node at (0, 4.7) {$M_{1/2} =$};
	\node at (0,4) {$\frac{x^4 + 2x^2 y^2 + y^4 + x^2 z^2}{yz^2}$};
	\node at (3.11, 3.8) {$M_{2/3}$};
	\node at (2.30, 5.76) {$M_{3/5}$}; 
	\node at (4.37, 1.98) {$M_{3/4}$}; 
	\node at (0.87, 6.97) {$M_{4/7}$}; 
	\node at (4.38, 5.85) {$M_{5/8}$}; 
	\node at (5.81, 3.35) {$M_{5/7}$}; 
	\node at (4.43, 0.34) {$M_{4/5}$}; 
	\node at (-3.11, 3.8) {$M_{1/3}$};
	\node at (-2.10, 5.76) {$M_{2/5}$}; 
	\node at (-4.37, 2.18) {$M_{1/4}$}; 
	\node at (-0.87, 6.97) {$M_{3/7}$}; 
	\node at (-4.38, 5.85) {$M_{3/8}$}; 
	\node at (-5.81, 3.35) {$M_{2/7}$}; 
	\node at (-4.43, 0.34) {$M_{1/5}$}; 
	
	\draw[thick] (0, -0.5) -- (0,2);
	
		\draw[thick] (0,2) -- (1.73,3);
			
			\draw[thick] (1.73, 3) -- (1.99, 4.48);	
	
				\draw[thick] (1.99, 4.48) -- (1.31, 5.53);
				
					\draw[thick] (1.31, 5.53) -- (0.36, 5.87);
					\draw[thick] (1.31, 5.53) --  (1.23, 6.53);
					
				\draw[thick] (1.99, 4.48) -- (2.98, 5.24);	
				
					\draw[thick] (2.98, 5.24) -- (3.24, 6.21);
					\draw[thick] (2.98, 5.24) -- (3.98, 5.24);
	
			\draw[thick] (1.73, 3) -- (3.14, 2.49);
	
				\draw[thick] (3.14, 2.49) -- (4.29, 2.97);
	
					\draw[thick] (4.29, 2.97) -- (4.79, 3.84);
					\draw[thick] (4.29, 2.97) -- (5.26, 2.71);
				
				\draw[thick] (3.14, 2.49) -- (3.72, 1.38);
					
					\draw[thick] (3.72, 1.38) -- (4.63, 0.96);
					\draw[thick] (3.72, 1.38) -- (3.55, 0.39);
					
		\draw[thick] (0,2) -- (-1.73,3);
		
			\draw[thick] (-1.73, 3) -- (-1.99, 4.48);	
		
				\draw[thick] (-1.99, 4.48) -- (-1.31, 5.53);
		
					\draw[thick] (-1.31, 5.53) -- (-0.36, 5.87);
					\draw[thick] (-1.31, 5.53) -- (-1.23, 6.53);
		
				\draw[thick] (-1.99, 4.48) -- (-2.98, 5.24);	
		
					\draw[thick] (-2.98, 5.24) -- (-3.24, 6.21);
					\draw[thick] (-2.98, 5.24) -- (-3.98, 5.24);
		
			\draw[thick] (-1.73, 3) -- (-3.14, 2.49);
		
				\draw[thick] (-3.14, 2.49) -- (-4.29, 2.97);
		
					\draw[thick] (-4.29, 2.97) -- (-4.79, 3.84);
					\draw[thick] (-4.29, 2.97) -- (-5.26, 2.71);
		
				\draw[thick] (-3.14, 2.49) -- (-3.72, 1.38);
		
					\draw[thick] (-3.72, 1.38) -- (-4.63, 0.96);
					\draw[thick] (-3.72, 1.38) -- (-3.55, 0.39);
	\end{tikzpicture}
\end{center}
\caption{Markov polynomials on the Conway topograph.} \label{fig:MLPoly}
\end{figure}
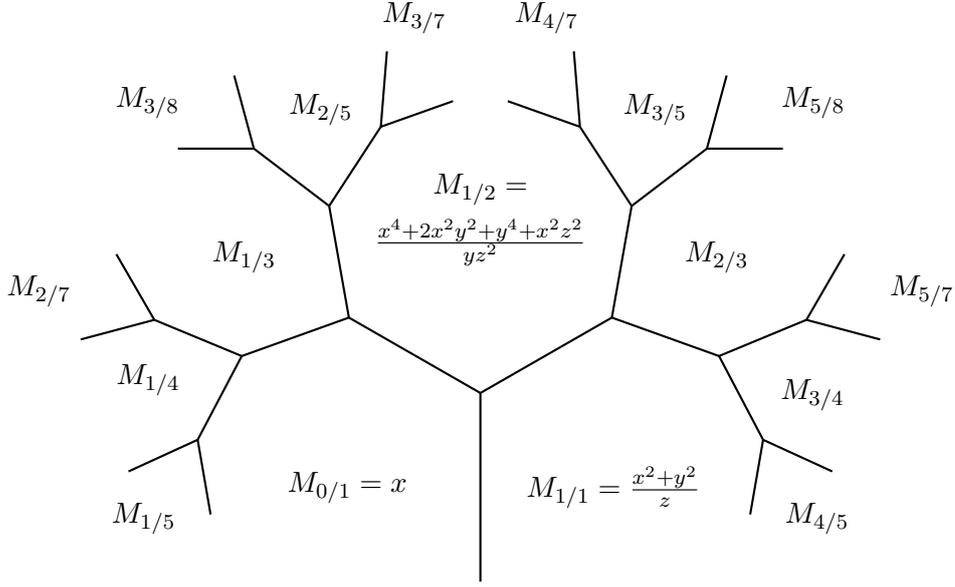


We will use the notation $M_{\rho}=M_\rho(x,y,z)$ to represent the Markov polynomial occupying the same region as the rational $\rho$ on their respective Conway topographs. 
Formally, this Frobenius parametrization is a mapping 
\beq{corres}
F: \rho \in [0, 1]\cap{\mathbb{Q}} \mapsto M_{\rho}(x, y, z). 
\eeq

Setting $x = y = z = 1$ in Markov polynomials gives the corresponding Markov number: $M_{\rho}(1, 1, 1) = m_{\rho}.$

We have the correspondence between Markov triples of numbers and the triples of Markov polynomials $M_{\rho}$ with $\rho \in [0, 1]\cap{\mathbb{Q}}$. In particular, the first Markov triples $(1,1,1),\, (1,1,2), \, (1,2,5)$ correspond to the triples $$(x,y,z),\, \left(x,y, \frac{x^2+y^2}{z}\right),\, \left(x, \frac{x^2+y^2}{z}, \frac{x^4 + 2x^2 y^2 + y^4 + x^2 z^2}{yz^2}\right)$$ respectively.
Thus we have the first few Markov polynomials
\begin{equation}
\label{eq:base_cases}
  M_{0/1} = x,\qquad M_{1/1}=\frac{x^2+y^2}z,\qquad 
M_{1/2} = \frac{(x^2+y^2)^2 + x^2z^2}{yz^2},\ldots
\end{equation}
Many of our proofs will proceed by induction on the Conway topograph, and
the solutions \eqref{eq:base_cases} will provide the base cases for the
induction.

To extend the Frobenius correspondence (\ref{corres}) to all rationals we can use the action of the symmetric group $S_3 \cong \mathrm{SL}_2(\mathbb{Z}_2)$:
\beq{symm}
M_{\rho}(x, y, z) = M_{\frac{1}{\rho}}(y, x, z), \quad
M_{\rho}(x, y, z) = M_{-\frac{1}{1+\rho}}(z, x, y).
\eeq

We believe that in this way we get all such solutions. More precisely, we have the following

\begin{conjecture}
 All the solutions $X,Y,Z \in \mathbb Z[x^{\pm 1}, y^{\pm 1}, z^{\pm 1}]$ of equation (\ref{MEq}) in integer Laurent polynomials, up to a change of sign of any two of them, are given by triples of Markov polynomials.
 \end{conjecture}

From now on we will focus only on the Markov polynomials corresponding to rationals in $[0, 1]$ as all others can be obtained by the above symmetries.

\section{Structure of Markov Polynomials}

The general structure of the Markov polynomial $M_{a/b}(x, y, z)$ can be described as follows.

\begin{theorem}
\label{denom}
Markov polynomials $M_{a/b}(x, y, z) $ with coprime $a,b>0$ have the following general form
\begin{equation}
\label{eqn:Rational Polynomials}
    M_{a/b}(x, y, z) = \frac{P_{a/b}(x^2, y^2, z^2)}{x^{a-1} y^{b-1} z^{a+b-1}},
\end{equation}
where $P_{a/b}(u,v,w)$ is a homogeneous polynomial of degree $a+b-1$ with non-negative integer coefficients, indivisible by $x$, $y$ or $z$.   
\end{theorem}

We set here $P_{0/1}\equiv 1$ to have the correct formula for $M_{0/1}(x,y,z)=x.$



\begin{proof}

The proof is by induction using the Vieta formula in the form \eqref{eq:vieta2}.
    
\begin{figure}[H]
\begin{center}
	\begin{tikzpicture} 
        \node at (0, -0.5) {$M_{\rho_Z},\ \rho_Z = \frac{a}{b}$};
        \node at (-1.5, 1.75) {$M_{\rho_X},\ \rho_X = \frac{c}{d}$};
        \node at (1.5, 1.75) {$M_{\rho_Y},\ \rho_Y = \frac{a+c}{b+d}$};
        \node at (0, 4) {$M_{\rho_{Z'}},\ \rho_{Z'} = \frac{a+2c}{b+2d}$};
 
        \draw[thick] (0, 1) -- (0, 2.5);
        \draw[thick](-1, 0) -- (0, 1);
        \draw[thick] (0, 1) -- (1, 0);
        \draw[thick] (-1, 3.5) -- (0, 2.5);
        \draw[thick] (0, 2.5) -- (1, 3.5);
\end{tikzpicture}
\end{center}
\caption{Section of the Conway topograph showing correspondence.} 
\label{fig:Assume}
\end{figure}
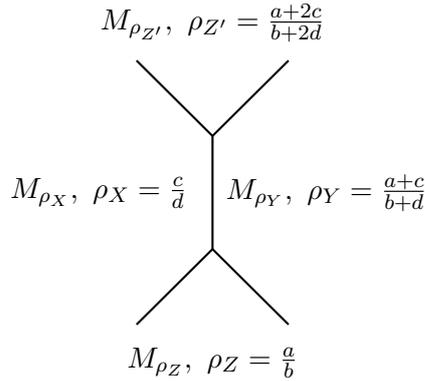

Consider the part of the Conway topograph shown in Figure~\ref{fig:Assume}.
We assume that the statement holds for $M_{\rho_X}$, $M_{\rho_Y}$ and
$M_{\rho_Z}$.  Then, according to \eqref{eq:vieta2},
\begin{align*}
  M_{\rho_{Z'}} &= k(x,y,z) M_{\rho_X} M_{\rho_Y} - M_{\rho_Z} \\
&= \frac{(x^2+y^2+z^2) P_{c/d}(x^2, y^2, z^2) P_{(a+c)/(b+d)}(x^2,y^2,z^2) - 
x^{2c}y^{2d}z^{2c+2d}P_{a/b}(x^2,y^2,z^2)}{x^{a+2c-1}y^{b+2d-1}z^{a+b+2c+2d-1}},
\end{align*}
where the denominator is of the required form.  We need only to be sure that 
there is
no cancellation of factors from the denominator, but for that to happen, one
of $P_{c/d}$ or $P_{(a+b)/(c+d)}$ would have to be divisible by $x$, $y$ or
$z$, which cannot happen.  We further see that the numerator satisfies the
required homogeneity property.
The base cases \eqref{eq:base_cases} satisfy \eqref{eqn:Rational Polynomials}
so the proof is complete.
  \end{proof}

As a consequence of the proof, we obtain a formula for the evolution of
the numerators along the Conway topograph:
\begin{equation}
  \label{eq:P_evolution}
  P_{\frac{a+2c}{b+2d}}= (u+v+w) 
P_{\frac{c}{d}}(u, v, w) P_{\frac{a+c}{b+d}}(u ,v ,w) - u^{c}v^{d}w^{c+d}P_{\frac{a}{b}}(u,v,w).
\end{equation}
  
  \subsection{Newton polygon of the numerator}

To study the numerator $P_{\rho}(u, v, w)$ of the Markov polynomial $M_{\rho}(x, y, z)$, $\rho=a/b$, it is natural to look first at the corresponding Newton polygon  $\Delta_{\rho}$, which can be considered as tropicalization of $P_{\rho}(u, v, w)$.

Whilst our polynomials consist of three variables $u, v, w$, by noting homogeneity of terms---
each term has total order in the numerator of $a+b-1$---we can simplify the problem (see Theorem~\ref{denom}). Indeed we can write $P_{\rho}$ as the following 
sum
\begin{equation}
\label{eqn:numerator}
    P_{\rho}(u, v, w) = \sum A_{ij} u^{i} v^{j} w^{a+b-1-i-j}.
\end{equation}
As a result of this homogeneity we can consider the plane projections $w = 1$, looking at the Newton polygon in the two variables $u, v$, and losing no additional information in the process. Note that this is now precisely the affine version ($w = 1$) of the Newton polygon for the polynomial $P_{\rho}(u, v, w)$.

For example, when $\rho = \frac{2}{3}, m_{\rho} = 29$ we have
\begin{displaymath}
    P_{\rho}(u, v, w) = u^4 + 4u^3 v + 6u^2 v^2 + 4u v^3 + v^4 + 2u^3 w + 5u^2 v w + 4u v^2 w + v^3 w + u^2 w^2,
\end{displaymath}
and setting $w = 1$ we obtain
\begin{equation*}
    P_{\rho}(u, v, w) = u^4 + 4u^3 v + 6u^2 v^2 + 4u v^3 + v^4 + 2u^3 + 5u^2 v + 4u v^2 + v^3 + u^2.
\end{equation*}

We define the Newton polygon of $M_{\rho}(x, y, z)$ as the convex hull 
\beq{Newton}
    \Delta_{\rho} = \Delta(P_{\rho}) := \Conv \lbrace (i, j): A_{ij} \neq 0 \rbrace \subset \mathbb{Z}^2.
\eeq

\begin{figure}[H]
    \begin{center}
    \begin{tikzpicture} [scale=0.6]
            \filldraw[blue!25] (8, 0) -- (0, 8) -- (0, 6) -- (4, 0) -- (8, 0);

            \draw[thick, ->] (0, 0) -- (0, 9);
            \draw[thick, ->] (0, 0) -- (9, 0);
            
            \filldraw [blue] (8, 0) circle (2pt); 
            \filldraw [blue] (6, 2) circle (2pt);
            \filldraw [blue] (4, 4) circle (2pt);
            \filldraw [blue] (2, 6) circle (2pt);
            \filldraw [blue] (0, 8) circle (2pt);

            \filldraw [blue] (6, 0) circle (2pt);
            \filldraw [blue] (4, 2) circle (2pt);
            \filldraw [blue] (2, 4) circle (2pt);
            \filldraw [blue] (0, 6) circle (2pt);
            
            \filldraw [blue] (4, 0) circle (2pt);

            \node at (8, -0.3) {4};
            \node at (6, -0.3) {3};
            \node at (4, -0.3) {2};
            \node at (2, -0.3) {1};
            \node at (0, -0.3) {0};

            \node at (-0.3, 8) {4};
            \node at (-0.3, 6) {3};
            \node at (-0.3, 4) {2};
            \node at (-0.3, 2) {1};
            \node at (-0.3, 0) {0};

            \node at (-0.3, 9) {$j$};
            \node at (9, -0.3) {$i$};
    \end{tikzpicture}
    \end{center}
    \caption{Newton polygon $\Delta_{\rho}$, of the Markov polynomial $M_{\rho}, \rho = \frac{2}{3}$.} \label{fig:NP29}
\end{figure}
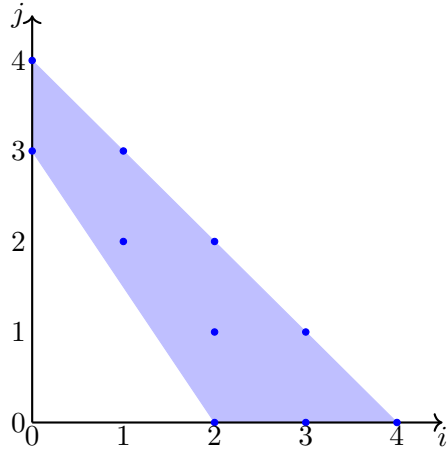

We have the following explicit description\footnote{We are very grateful to Alexey Ustinov for the stimulating discussions of the Newton polygons for Markov polynomials, which led to formula (\ref{eqn:Newton Poly}).} of the Newton polygons of the numerators of Markov polynomials in general case.\footnote{As we were informed by Ralf Schiffler, this result can be derived from the general description of Newton polygons for rank 3 cluster algebras \cite{LLS}.} 


\begin{theorem}
\label{Newton Poly}
    Given a rational $\rho=\frac{a}{b}$, the Newton polygon $\Delta_{a/b}$ of the numerator of Markov polynomial $M_{\rho}(x, y, z)$  is the area on the $ij$-plane with $i, j \geq 0$ satisfying the conditions
    \begin{equation}
    \label{eqn:Newton Poly}
         \left\{
        \begin{aligned}
            & \frac{i}{a} + \frac{j}{b} \geq 1 \\
            & i + j \leq a+b-1.
        \end{aligned}
        \right.
    \end{equation}
\end{theorem}

The proof is by induction. The first few cases are easily verifiable by direct computations.
For the inductive step, we use the Vieta formula $ZZ' = X^2 + Y^2$ and verify that the Minkowski sum of the Newton polygons on the LHS coincides with the convex hull of the union of the scaled (due to squared terms) polygons on the RHS, taking care with the effect of the denominators on the subsequent numerator during the addition on the RHS. 
This is straightforward but a bit technical, so we omit the details.

We know due to \cite{Pro}  that  Markov polynomials have non-negative coefficients, but some of the terms corresponding to the lattice points in the Newton polygon $\Delta_{\rho}$, in principle, could be zero. We conjecture that this is not the case and all the corresponding coefficients are positive.


\begin{conjecture} (Saturation conjecture)
\label{Saturation}
 The   terms that appear in the numerator of a Markov polynomial $M_{\rho}$ are precisely those corresponding to the integer lattice points in the Newton polygon $\Delta_{\rho}$. 
\end{conjecture}

Note that in the original $(x,y,z)$ variables the Newton polygon is twice 
larger and saturation does not hold.

We will prove Conjecture~\ref{Saturation} in certain special cases of Markov 
polynomials (see Section~\ref{sec:funf}).

We should also mention in this relation an important result by Fei \cite{Fei}, who proved saturation for the $F$-polynomials  of a rigid representation of any finite-dimensional basic algebra.


\section{Coefficients of Markov polynomials}

We can view the coefficients $A_{ij}$ of Markov polynomial 
\begin{equation*}
    M_{a/b}(x, y, z) = \frac{1}{x^{a-1} y^{b-1} z^{a+b-1}} \sum_{i + j \leq a+b-1, \frac{i}{a} + \frac{j}{b} \geq 1} A_{ij} x^{2i} y^{2j} z^{2(a+b-1-i-j)}
\end{equation*}
as the weights on the lattice points of Newton polygon.


For example, for $\rho = \frac{2}{3}$ the `weighted' polygon, $\Delta_{\rho}$ is given in Figure \ref{fig:wNP29}.
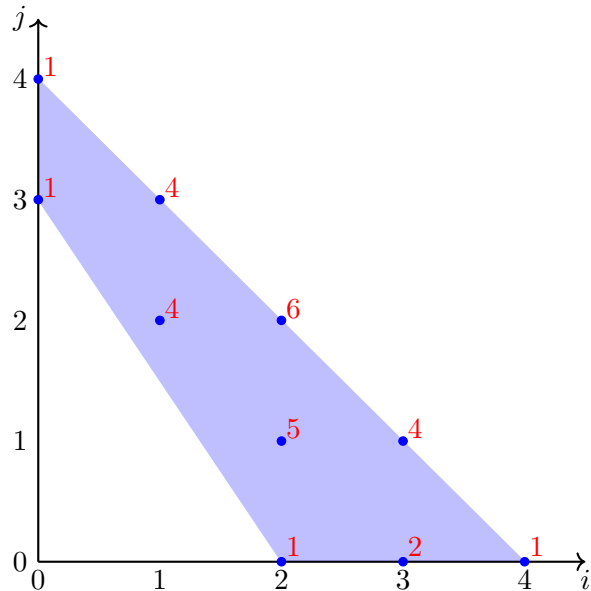
\begin{figure}[H]
    \begin{center}
    \begin{tikzpicture}  [scale=0.8]
            \filldraw[blue!25] (8, 0) -- (0, 8) -- (0, 6) -- (4, 0) -- (8, 0);

            \draw[thick, ->] (0, 0) -- (0, 9);
            \draw[thick, ->] (0, 0) -- (9, 0);
            
            \filldraw [blue] (8, 0) circle (2pt);
            \node[red] at (8.2, 0.25) {1};
            \filldraw [blue] (6, 2) circle (2pt);
            \node[red] at (6.2, 2.2) {4};
            \filldraw [blue] (4, 4) circle (2pt);
            \node[red] at (4.2, 4.2) {6};
            \filldraw [blue] (2, 6) circle (2pt);
            \node[red] at (2.2, 6.2) {4};
            \filldraw [blue] (0, 8) circle (2pt);
            \node[red] at (0.2, 8.2) {1};

            \filldraw [blue] (6, 0) circle (2pt);
            \node[red] at (6.2, 0.25) {2};
            \filldraw [blue] (4, 2) circle (2pt);
            \node[red] at (4.2, 2.2) {5};
            \filldraw [blue] (2, 4) circle (2pt);
            \node[red] at (2.2, 4.2) {4};
            \filldraw [blue] (0, 6) circle (2pt);
            \node[red] at (0.2, 6.2) {1};
            
            \filldraw [blue] (4, 0) circle (2pt);
            \node[red] at (4.2, 0.25) {1};

            \node at (8, -0.3) {4};
            \node at (6, -0.3) {3};
            \node at (4, -0.3) {2};
            \node at (2, -0.3) {1};
            \node at (0, -0.3) {0};

            \node at (-0.3, 8) {4};
            \node at (-0.3, 6) {3};
            \node at (-0.3, 4) {2};
            \node at (-0.3, 2) {1};
            \node at (-0.3, 0) {0};

            \node at (-0.3, 9) {$j$};
            \node at (9, -0.3) {$i$};
    \end{tikzpicture}
    \end{center}
    \caption{`Weighted' Newton polygon of the Markov polynomial $M_{2/3}$.} \label{fig:wNP29}
\end{figure}
Looking at the coefficients that appear on the lines $i = 0, j = 0$
and $i + j = a + b - 1 = 4$, we notice that they are binomial coefficients. 
This is no coincidence as we prove below.

Expanding the Markov polynomials as Laurent series in $z,y,x$ leads respectively to
$$
    M_{a/b}(x, y, z) = \frac{1}{x^{a-1} y^{b-1}} \left[\frac{T_0(x^2, y^2)}{z^{a+b-1}} + \frac{T_1(x^2, y^2)}{z^{a+b-3}} + \dots + \frac{T_k(x^2, y^2)}{z^{a+b-1-2k}} + \cdots\right]
    $$
    $$
    = \frac{1}{x^{a-1} z^{a+b-1}} \left[\frac{R_0(x^2, z^2)}{y^{b-1}} + \frac{R_1(x^2, z^2)}{z^{b-3}} + \dots + \frac{R_k(x^2, z^2)}{y^{b-1-2k}} + \cdots\right] 
    $$
    $$
    = \frac{1}{y^{b-1} z^{a+b-1}} \left[\frac{S_0(y^2, z^2)}{x^{a-1}} + \frac{S_1(y^2, z^2)}{x^{a-3}} + \dots + \frac{S_k(y^2, z^2)}{x^{a-1-2k}} + \cdots\right]
$$
for certain polynomials $T_i, R_i, S_i$.
These polynomials with $i=0$  contain information about the coefficients on the boundary of the Newton polygon ($S_0$ on the vertical, $R_0$ on the horizontal and $T_0$ on the upper diagonal line), while $i=1,2,\ldots$ correspond to the neighbouring parallel lines.

\begin{theorem}
In the following cases the polynomials $S_i,R_i,T_i$ can be written explicitly 
as:
\begin{itemize}
    \item $S_0(v, w) = v^b (v + w)^{a-1}$ \\
    \item $R_0(u, w) = u^a (u + w)^{b-1}$ \\
    \item $R_1(u, w) = u^a(3a - 1) (u + w)^{b - 2} + u^{a+1} (b - 2a) (u + w)^{b -3}$ \\
    \item $T_0(u, v) = (u + v)^{a+b-1}$ \\
    \item $T_1(u, v) = (a - 1) (u + v)^{a + b - 2} + u (b - a) (u + v)^{a + b -3} $ \\
    \item $
T_2(u, v) = \frac{(a - 1)(a - 2)}{2} (u + v)^{a+b-3} + u [a(b - a) - a] (u + v)^{a+b-4} \\
       \qquad+ \frac{1}{2} u^2 [(b-a)^2 + 5a - 3b] (u + v)^{a+b-5}.     $
\end{itemize}
\end{theorem}

\begin{proof}
    We will prove only the formula for $T_0$; all other cases are similar.

It is clear from \eqref{eq:base_cases} that $P_{0/1}, P_{1/1}, P_{1/2}$ all 
have the required form, so let us consider again the section of the
Conway topograph in Figure~\ref{fig:Assume}, and suppose that
\begin{align*}
  P_{\rho_X}(u,v,w) = T_{0,X}(u,v) + \Ord(w) &= (u+v)^{c+d-1} + \Ord(w) \\
  P_{\rho_Y}(u,v,w) = T_{0,Y}(u,v) + \Ord(w) &= (u+v)^{a+b+c+d-1} + \Ord(w)\\
    P_{\rho_Z}(u,v,w) = T_{0,Z}(u,v) + \Ord(w)& = (u+v)^{a+b-1} + \Ord(w),
 \end{align*} 
where $\Ord(w)$ denotes a polynomial function of $u,v,w$ with $w$ as a 
factor.

From \eqref{eq:P_evolution} we have
\begin{align*}
  P_{\rho_{Z'}}(u,v,w) &= (u+v+w)P_{\rho_X}(u,v,w) P_{\rho_Y}(u,v,w) -
u^cv^dw^{c+d} P_{\rho_Z}(u,v,w) \\
&=(u+v)(u+v)^{c+d-1}(u+v)^{a+b+c+d-1} + \Ord(w) \\
&=(u+v)^{a+b+2c+2d-1} + \Ord(w),
\end{align*}
giving us the expected form for $P_{\rho_{Z'}}(u,v,w)$ with
$\rho_{Z'} = (a+2c)/(b+2d).$
\end{proof}

As a corollary we have an explicit form of the following coefficients
(see Figure~\ref{fig:Coeff}).

\begin{theorem}
\label{2nd Hor}
\label{Coefficients}
    For the Markov polynomial  $M_{a/b}(x, y, z)$ we know the coefficients $A_{ij}$ in the following cases:
    \begin{itemize}
        \item $\displaystyle A_{0,j} = \binom{a-1}{j-b}$ \\
        \item $\displaystyle A_{i,0} = \binom{b-1}{i-a}$ \\
        \item $\displaystyle A_{i, 1} = (3a - 1) \binom{b-2}{i-a} + (b - 2) \binom{b-3}{i-a-1}$ \\
        \item If $i + j = a + b - 1$ then
$\displaystyle A_{i,j} = \binom{a+b-1}{i}$,
        \item If $i + j = a + b - 2$ then
$$A_{i,j} = (a - 1) \binom{a+b-2}{i} + (b - a) \binom{a+b-3}{i-1}$$
        \item If $i + j = a + b - 3$ then
\begin{align*}
A_{i,j} &=
                \frac{(a - 1)(a - 2)}{2} \binom{a+b-3}{i} + [a(b-a) - a] \binom{a+b-4}{i-1} \\ &\qquad + \frac{1}{2}[(b-a)^2 + 5a - 3b] \binom{a+b-5}{i-2} 
\end{align*}
\end{itemize}
\end{theorem}

In particular, the boundary coefficients are binomials and
we have an explicit expression for some of the interior coefficients
as sums of binomial coefficients. 

\begin{figure}[H]
\begin{center}
    \begin{tikzpicture}[scale=0.4]
            \draw[blue!50] (10, 0) -- (0, 10) -- (0, 7) -- (4, 0) -- (10, 0);
            \filldraw[red!50] (10.1, 0) -- (0, 10.1) -- (0, 9.9) -- (9.9, 0) -- (10.1, 0);
            \filldraw[red!50] (9.6, 0) -- (0, 9.6) -- (0, 9.4) -- (9.4, 0) -- (9.6, 0);
            \filldraw[red!50] (9.1, 0) -- (0, 9.1) -- (0, 8.9) -- (8.9, 0) -- (9.1, 0);
            \filldraw[red!50] (4, -0.1) -- (4, 0.1) -- (10, 0.1) -- (10, -0.1) -- (4, -0.1);
            \filldraw[red!50] (4, 0.4) -- (4, 0.6) -- (9.5, 0.6) -- (9.5, 0.4) -- (4, 0.4);
            \filldraw[red!50] (-0.1, 7) -- (0.1, 7) -- (0.1, 10) -- (-0.1, 10) -- (-0.1, 7);

            \draw[thick, ->] (0, 0) -- (0, 10.5);
            \draw[thick, ->] (0, 0) -- (10.5, 0);
    \end{tikzpicture}
    \end{center}
    \caption{Lines on the Newton polygon where the coefficients are explicitly known.} \label{fig:Coeff}
\end{figure}
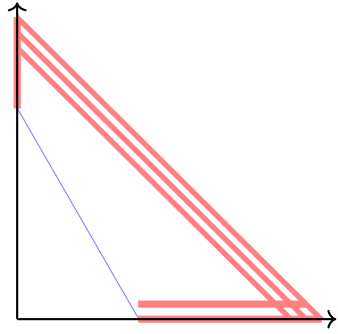

For the rationals of the form $\rho=\frac{1}{n}$ and $\rho=\frac{2}{2n-1}$ we can add also the coefficients on one more vertical line with $i=1.$

\begin{theorem}
\label{Fibonacci Critical}
 For the Markov polynomials  $M_{1/n}$ and $M_{2/(2n-1)}$ we have respectively
$$
        S_1(y^2, z^2) = \sum_{k = 1}^n k y^{2(k-1)}z^{2(n-k)}
       $$  
 and   
$$
        S_1(y^2, z^2) 
        = 2n y^{4n-2} + \sum_{k=1}^{n-1} 4k y^{2(n+k-1)} z^{2(n-k)}.
$$
\end{theorem}

The proof of these results is similar to that of Theorems 3.1 and 4.1 and is based on relation (\ref{eq:P_evolution}).

In particular, on the line $i=1$ the Markov polynomial  $M_{1/n}(x, y, z)$ has the coefficients $1,2,\dots, n$ (see Figure \ref{fig:NP89} for the example with $n = 5$, corresponding to the Markov number $89$).

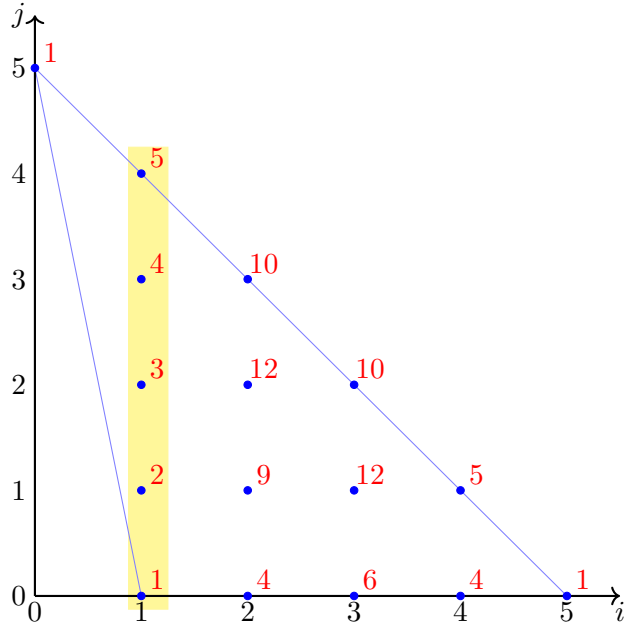
\begin{figure}[H]
    \begin{center}
    \begin{tikzpicture}[scale=0.7]
    \filldraw[yellow!50] (1.75, -0.25) -- (2.5, -0.25) -- (2.5, 8.5) -- (1.75, 8.5);
            
            \draw[blue!50] (10, 0) -- (0, 10) -- (2, 0) -- (10, 0);

            \draw[thick, ->] (0, 0) -- (0, 11);
            \draw[thick, ->] (0, 0) -- (11, 0);

            \filldraw [blue] (10, 0) circle (2pt);
            \node[red] at (10.3, 0.3) {1};
            \filldraw [blue] (8, 2) circle (2pt);
            \node[red] at (8.3, 2.3) {5};
            \filldraw [blue] (6, 4) circle (2pt);
            \node[red] at (6.3, 4.3) {10};
            \filldraw [blue] (4, 6) circle (2pt);
            \node[red] at (4.3, 6.3) {10};
            \filldraw [blue] (2, 8) circle (2pt);
            \node[red] at (2.3, 8.3) {5};
            \filldraw [blue] (0, 10) circle (2pt);
            \node[red] at (0.3, 10.3) {1};
            
            \filldraw [blue] (8, 0) circle (2pt);
            \node[red] at (8.3, 0.3) {4};
            \filldraw [blue] (6, 2) circle (2pt);
            \node[red] at (6.3, 2.3) {12};
            \filldraw [blue] (4, 4) circle (2pt);
            \node[red] at (4.3, 4.3) {12};
            \filldraw [blue] (2, 6) circle (2pt);
            \node[red] at (2.3, 6.3) {4};

            \filldraw [blue] (6, 0) circle (2pt);
            \node[red] at (6.3, 0.3) {6};
            \filldraw [blue] (4, 2) circle (2pt);
            \node[red] at (4.3, 2.3) {9};
            \filldraw [blue] (2, 4) circle (2pt);
            \node[red] at (2.3, 4.3) {3};
            
            \filldraw [blue] (4, 0) circle (2pt);
            \node[red] at (4.3, 0.3) {4};
            \filldraw [blue] (2, 2) circle (2pt);
            \node[red] at (2.3, 2.3) {2};

            \filldraw [blue] (2, 0) circle (2pt);
            \node[red] at (2.3, 0.3) {1};

            \node at (10, -0.3) {5};    
            \node at (8, -0.3) {4};
            \node at (6, -0.3) {3};
            \node at (4, -0.3) {2};
            \node at (2, -0.3) {1};
            \node at (0, -0.3) {0};

            \node at (-0.3, 10) {5};
            \node at (-0.3, 8) {4};
            \node at (-0.3, 6) {3};
            \node at (-0.3, 4) {2};
            \node at (-0.3, 2) {1};
            \node at (-0.3, 0) {0};

            \node at (-0.3, 11) {$j$};
            \node at (11, -0.3) {$i$};
    \end{tikzpicture}
    \end{center}
   \caption{Weighted Newton polygon of $M_{1/5}$ with the highlighted coefficients on the line $i=1$.} \label{fig:NP89}
\end{figure}

In the next two sections we consider in more detail two special series of Markov polynomials $M_\rho$ with $\rho=1/n$ and $\rho=n/(n+1)$.

\section{Markov polynomials corresponding to Fibonacci and Pell numbers}
\label{sec:funf}
\subsection{Markov-Fibonacci polynomials}
\label{sec:Fib Poly}
Fibonacci numbers satisfy the following recurrence
\begin{equation}
\label{Fib Rec}
    F_{n+1} = F_{n} + F_{n-1},
\end{equation}
subject to the intial conditions $F_0 = 0, F_1 = 1$ (sequence
A000045 of \cite{OEIS}).  The first few terms are:
 $$0, 1, 1, 2, 3, 5, 8, 13, 21, 34, 55, 89, 144, 233, 377, 610, \dots. $$

It is known that the odd-indexed Fibonacci numbers $F_{2k+1}$ are also Markov numbers \cite{Aig}. 
Indeed, the odd-indexed Fibonacci numbers satisfy 
\begin{equation*}
    F_{2k+1} = 3F_{2k-1} - F_{2k-3},
\end{equation*}
which means that $Y_k= F_{2k-1}$ satisfy the Vieta recursion $Y_{k+1}=3Y_k-Y_{k-1}$ for the Markov triples $(1, Y_{k-1}, Y_k).$

We will call the Markov polynomials $M_{1/n}(x,y,z)$ corresponding to the Fibonacci numbers $F_{2n-1}$ the {\it Markov-Fibonacci polynomials.} They can be viewed as $3$-parameter quantisations of the odd-indexed Fibonacci numbers.

We should note that a closely related notion was introduced by  Caldero and Zelevinsky \cite{CZ} as cluster variables of the $A_1^{(1)}$ cluster algebra \textit{via} the recurrence relation
\begin{equation}
\label{A11 Cluster}
    f_{m+1} = \frac{f_m^2 + 1}{f_{m-1}}.
\end{equation}
The {\it Caldero-Zelevinsky Fibonacci polynomial } $f_m$ is defined as the corresponding $f_m(x_1,x_2)$ considered as the function of the two initial cluster variables $f_1=x_1, f_2=x_2.$ 


\begin{theorem}[\cite{CZ}, \cite{Zel}]
\label{Thm CZ Fib}
    For every $n \geq 0$ there are the explicit formul\ae
    \begin{align*}
     f_{2n}  &= \frac{1}{x_1^n x_2^n} \sum_{q+r \leq n} \binom{n-r}{q} \binom{n-q}{r} x_1^{2q} x_2^{2r}, \\
   f_{2n+1} &= \frac{1}{x_1^{n+1} x_2^n} \left(x_2^{2(n+1)} + \sum_{q+r\leq n} \binom{n-r}{q} \binom{n+1-q}{r} x_1^{2q} x_2^{2r}\right).
    \end{align*}
    In particular $f_m$ are Laurent polynomials in $x_1, x_2$ with positive coefficients.
\end{theorem}


Comparing Eqs.~\eqref{A11 Cluster} with the Vieta recurrence for Markov polynomials
\begin{equation}
\label{Fibonacci Vieta}
    M_{1/(k+1)} = \frac{M_{1/k} + x^2}{M_{1/(k-1)}}
\end{equation}
we see that $M_{1/k}$'s are a homogeneous version of the cluster variables $f_m$:
\begin{equation}
\label{twofib}
    f_m(x_1, x_2) = M_{1/m}(1, x_2, x_1).
\end{equation}

Using Theorem~\ref{Thm CZ Fib}, we have

\begin{theorem}
\label{Fibonacci Poly Coeff}
    The Markov-Fibonacci polynomials  $M_{1/(n+1)}(x,y,z)$ have coefficients
    \begin{equation}
    \label{eqn:Fib Coeff}
        A_{ij} = \binom{n-j}{n+1-i-j} \binom{i+j}{j},
    \end{equation}
    where $A_{ij}$ is the coefficient of the numerator monomial $x^{2i} y^{2j} z^{2(n+1-i-j)}$.
\end{theorem}

\begin{corollary}
\label{Fibonacci Saturation}
Markov-Fibonacci polynomials satisfy the Saturation Conjecture \ref{Saturation}.
\end{corollary}

Another simple consequence of this result is the first part of Theorem \ref{Fibonacci Critical} since setting $i = 1$ leads to
$$
 A_{1j} = \binom{n-j}{n-j} \binom{1+j}{j}= j+1.
$$

It is interesting to note\footnote{This was observed by Sam Evans after discussions with Valentin Ovsienko.}  that our Markov-Fibonacci polynomials $M_{1/n}(u,v,w)$ are related to Fibonacci polynomials $\mathcal F_n(q)$ from the theory of $q$-rational numbers developed by Morier-Genoud and Ovsienko \cite{MGO}.

The Fibonacci polynomials $\mathcal{F}_n(q)$ are the denominators of the $q$-deformed rationals $\left[\frac{F_{n+1}}{F_{n}}\right]_q$, where $F_n$ are the Fibonacci numbers \cite{LMGOV}. They can be determined by the recurrence relation
	\begin{equation}
	\label{quantFibrec}
	\mathcal{F}_{n+2} (q) = [3]_q \mathcal{F}_n (q) - q^2 \mathcal{F}_{n-2} (q),
	\end{equation} 
	where $[3]_q = 1 + q + q^2$, and the initial conditions 
\[\mathcal{F}_0(q) = 1,\; \mathcal{F}_2(q) = 1+q \qquad 
\text{and} \qquad 
\mathcal{F}_1(q) = 1,\; \mathcal{F}_3(q) = 1+q+q^2.\]

\begin{proposition}
The Fibonacci polynomials $\mathcal{F}_n(q)$ are the following specialisations
	\begin{equation}
	\label{relationfib}
	\mathcal{F}_{2n} (q) = P_{1/n}(q,1,q^2)
	\end{equation} 
	of the numerators $P_{1/n}(u,v,w)$ of the Markov polynomials $M_{1/n}(u,v,w)$.
\end{proposition}

\begin{proof}
This follows from the comparison of the recurrence relations (\ref{eq:P_evolution}) and (\ref{quantFibrec}), and $P_{1/0}(u,v,w)=1$, $P_{1/1}(u,v,w)=u+v$.
\end{proof}



\subsection{Markov-Pell Polynomials}

We now turn to the Markov polynomials corresponding to Pell numbers,
satisfying the recurrence
\begin{equation}
\label{Pell Rec}
    P_{n+1} = 2P_{n} + P_{n-1},
\end{equation}
subject to the intial conditions $P_0 = 0, P_1 = 1$ (sequence
A000129 of \cite{OEIS}). 
The first few Pell numbers are $0, 1, 2, 5, 12, 29, 70, 169, \dots$.

It is known \cite{Aig} that the odd-indexed Pell numbers $P_{2k+1}$ are also Markov numbers, since they satisfy the recurrence
$$
P_{2k+1}=6P_{2k-1}-P_{2k-3},
$$
equivalent to Vieta recursion for Markov triples $(2, Y_{k-1}, Y_k), \, Y_k=P_{2k-1}.$

The Markov polynomials $M_{k/(k+1)}$ can be viewed as a quantization of the Pell numbers $P_{2k+1}$.

We introduce now the {\it  Markov-Pell  polynomials} \textit{via} the following recurrences, (which are quantized versions of (\ref{Pell Rec}))
\begin{equation}
\label{Pell Rec OddEven}
  \begin{aligned}
    R_{2k+1} &= (x^2 + y^2) R_{2k} + x^2 z^2 R_{2k-1}, \\
    R_{2k} &= (x^2 + y^2) R_{2k-1} + y^2 z^2 R_{2k-2},
  \end{aligned}
\end{equation}
with $R_0=0, \, R_1 = 1, \, R_2=x^2+y^2, R_3 = x^4 + 2x^2y^2 + y^4 + x^2z^2$.


The odd-indexed Markov-Pell  polynomials satisfy the recurrence
\begin{equation}
\label{Mar Pell 2}
    R_{2k+1} = (x^2 + y^2)(x^2 + y^2 + z^2) R_{2k-1} - x^2 y^2 z^4 R_{2k-3},
\end{equation}
which is precisely the recurrence for the numerators of the Markov polynomials $M_{k/(k+1)}$ (and because of the special  initial data they coincide with these numerators).

From the equations \eqref{Pell Rec OddEven} we can deduce
\begin{corollary}
\label{Pell Saturation}
 Markov-Pell polynomials satisfy the Saturation Conjecture \ref{Saturation}.
\end{corollary}

The equations ~\eqref{Mar Pell 2} imply the recurrence for the corresponding coefficient $A_{i,j}^{(2k+1)}$ at the monomial $x^{2i} y^{2j} z^{2(k-i-j)}$ in $R_{2k+1}:$
\begin{equation}
\label{Pell Coeff Rec}
    A_{i, j}^{(2k+1)} = A_{i-2, j}^{(2k-1)} + 2 A_{i-1, j-1}^{(2k-1)} + A_{i, j-2}^{(2k-1)}+ A_{i-1, j}^{(2k-1)} + A_{i, j-1}^{(2k-1)} - A_{i-1, j-1}^{(2k-3)}.
\end{equation}


We can also produce the Binet-type formula for the corresponding Markov-Pell polynomials, similar to the classical Binet formula for the Fibonacci numbers
\begin{equation*}
    F_n = \frac{1}{\sqrt{5}}\left(\left(\frac{1+\sqrt{5}}{2}\right)^n - \left(\frac{1-\sqrt{5}}{2}\right)^n\right).
\end{equation*}

%

Indeed, we can rewrite Eqs.~\eqref{Pell Rec OddEven} in  matrix form as
\begin{equation*}
    \begin{pmatrix}
        R_{2k+1} \\
        R_{2k}
    \end{pmatrix}
    =
    \begin{pmatrix}
        (x^2 + y^2)^2 + x^2 z^2 & (x^2 + y^2) x^2 z^2 \\
        (x^2 + y^2) &  y^2 z^2
    \end{pmatrix}
    \begin{pmatrix}
        R_{2k-1} \\
        R_{2k-2}
    \end{pmatrix}
\end{equation*}
The characteristic equation of the corresponding matrix is
\begin{equation}
\label{Pell characteristic}
    \lambda^2 - (x^2 + y^2)(x^2 + y^2 + z^2) \lambda + x^2 y^2 z^4 = 0
\end{equation}
giving the eigenvalues
\begin{equation*}
    \lambda_{\pm}(x,y,z) = \frac{1}{2} \left[(x^2 + y^2)(x^2 + y^2 + z^2) \pm \sqrt{\mathcal{D}}\right]
\end{equation*}
where $${\mathcal{D}} = (x^2 + y^2)^4 + 2z^2(x^2 + y^2)^3 + z^4(x^2 - y^2)^2$$ is the discriminant of \eqref{Pell characteristic}.

The standard procedure leads now to the following Binet-type formula for the Markov-Pell polynomials:
\begin{equation}
\label{Binet}
R_{2k+1} = \frac{1}{\sqrt{\mathcal{D}}} \left( (\lambda_+ - y^2z^2) \lambda_+^k
- (\lambda_--y^2 z^2) \lambda_-^k \right).
\end{equation}

\section{Log-Concavity}
A large number of naturally-occurring combinatorial sequences exhibit the
\emph{logarithmically concave} (log-concave) property \cite{Sta}.
Further to saturation, experimental data that we have looked-at so far
has shown that coefficients of the Markov polynomials are log-concave
in all directions.
\begin{definition}
    A sequence $x = (x_0, x_1, \dots, x_n)$ is said to be \textit{log-concave} if it satisfies the property
    \begin{equation}
    \label{logcon}
        x_k^2 \geq x_{k-1} x_{k+1},
    \end{equation}
    for $k \in \lbrace 1, 2, \dots, n-1 \rbrace$. 
\end{definition}

A well-known example is the binomial sequence $x_k=\binom{n}{k}, \, k=0,\dots,n$:
$$
\frac{x_k^2}{x_{k-1}x_{k+1}}=\binom{n}{k}^2/\binom{n}{k-1}\binom{n}{k+1}=\frac{k+1}{k}\frac{n-k+1}{n-k}>1.
$$

\begin{definition}
    We say that the weights on the Newton polygon satisfy the \textit{weak log-concavity} property if their sequences in all principal directions (horizontal, vertical and diagonal with $i + j = \text{constant}$) are log-concave.
\end{definition}

\begin{conjecture} (Log-concavity conjecture)
\label{Markovlogconcave}
 The coefficients of Markov polynomials are weakly log-concave on the corresponding Newton polygon.
\end{conjecture}

Note that the Conjecture \ref{Markovlogconcave} implies the Saturation Conjecture \ref{Saturation}, claiming the positivity of all the coefficients inside the Newton polygon.
Indeed, we know that the coefficients on the boundary are positive, so the appearance of internal zero coefficients will contradict the log-concavity property (\ref{logcon}).

\begin{theorem}
\label{log-concave_Fib}
    Conjecture 6.3 holds for the Markov-Fibonacci polynomials.
    \end{theorem}

 \begin{proof}
    We need to check the log-concavity condition in all three principal directions. Using \eqref{eqn:Fib Coeff}, in the horizontal direction we have

     $$   A_{i, j}^2 = \binom{n-j}{n+1-i-j}^2\binom{i+j}{j}^2 = \left(\frac{(n-j)!^2}{(n+1-i-j)!^2(i-1)!^2}\right) \left(\frac{(i+j)!^2}{i!^2 j!^2}\right), 
     $$
            \begin{align*}
        A_{i-1, j} A_{i+1, j} &= \binom{n-j}{n+2-i-j} \binom{n-j}{n-i-j} \binom{i+j-1}{j} \binom{i+j+1}{j} \\
        &= \left(\frac{(n-j)!^2}{(n+2-i-j)!(i-2)!(n-i-j)!i!}\right) \left(\frac{(i-1+j)! (i+1+j)!}{j!^2(i-1)!(i+1)!}\right).
    \end{align*}
    This implies that
    \begin{align*}
        \frac{A_{i, j}^2}{A_{i-1, j} A_{i+1, j}} &= \left(\frac{n+2-i-j}{n+1-i-j}\right) \left(\frac{i}{i-1}\right) \left(\frac{i+j}{i+j+1}\right) \left(\frac{i+1}{i}\right) \\
        &= \left(\frac{n+2-i-j}{n+1-i-j}\right) \left(\frac{i}{i-1}\right) \left(\frac{i(i+1) + j(i+1)}{i(i+1) + ji}\right).
    \end{align*}
    We see that all three terms here are greater than $1$ and so we have the required inequality
$
        A_{i,j}^2 > A_{i-1, j}A_{i+1, j}.
$

    Similarly in the vertical case we have
  $$
        \frac{A_{i, j}^2}{A_{i, j-1} A_{i, j+1}} = \left[\left(\frac{n-j}{n+1-j}\right) \left(\frac{n+2-i-j}{n+1-i-j}\right)\right] \left[\left(\frac{i+j}{i+j+1}\right) \left(\frac{j+1}{j}\right)\right]
        $$
        $$
        = \left[\frac{(n+1-i-j)(n-j) + (n-j)}{(n+1-i-j)(n-j) + (n+1-i-j)}\right] \left[\frac{j(i+j) + (i+j)}{j(i+j) + j}\right],
$$
    with both of the final fractions being greater than $1$ (the first requires that $i \geq 1$ but the vertical line for $i = 0$ consists of only a single point, so can be ignored).

    In the diagonal case the second binomials in  $A_{i-1, j+1},  A_{i, j}, A_{i+1, j-1}$ are respectively
    $
  \binom{i+j}{j+1}, \, \binom{i+j}{j}, \,  \binom{i+j}{j-1},
$
    which are consecutive binomial coefficients satisfying the log-concavity property.
    As a result we have
$$
        \frac{A_{i, j}^2}{A_{i-1, j+1} A_{i+1, j-1}} \geq \left(\frac{(n-j)!^2}{(n-j+1)!(n-j-1)!}\right) \left(\frac{(n-i-j)!(n+2-i-j)!}{(n+1-i-j)!^2}\right)$$
        $$
      = \left(\frac{n-j}{n-j+1}\right) \left(\frac{n+2-i-j}{n+1-i-j}\right) = \frac{(n-j)(n+1-i-j) + (n-j)}{(n-j)(n+1-i-j) + (n+1-i-j)},
$$
    which is greater than or equal to $1$ as required.
 \end{proof}
 \medskip

The Conjecture \ref{Markovlogconcave} holds along the boundaries of the polygon, since these coefficients are precisely binomial coefficients, which are known to be log-concave. We will show now that this holds also on the $2$nd horizontal line and upper diagonal of the Newton polygon.

\begin{lemma}
    Sequences of the form
    \begin{equation*}
        x_k = A\binom{c}{k} + B\binom{c}{k-1}, \,\, k=0,1,\dots, c
    \end{equation*}
    for positive $A, B, c$ are always log-concave.
\end{lemma}

\begin{proof}
We have
    \begin{equation*}
        x_k^2 = A^2 \binom{c}{k}^2 + 2AB \binom{c}{k} \binom{c}{k-1} + B^2 \binom{c}{k-1}^2,
    \end{equation*}
    \begin{multline*}
        x_{k-1} x_{k+1} = A^2 \binom{c}{k-1} \binom{c}{k+1} + AB \left[\binom{c}{k-1} \binom{c}{k} + \binom{c}{k-2} \binom{c}{k+1} \right] \\+ B^2 \binom{c}{k-2} \binom{c}{k}.
    \end{multline*}
 It is clear that the coefficients of $A^2$ and $B^2$ are greater in the case of $x_k^2$ due to the log-concavity of the binomial coefficients. So if we can show that
     \begin{equation*}
         \binom{c}{k-1} \binom{c}{k} \geq \binom{c}{k-2} \binom{c}{k+1},
     \end{equation*}
     then the result will follow.
Since 
     \begin{equation*}
         \frac{\binom{c}{k-1} \binom{c}{k}}{\binom{c}{k-2} \binom{c}{k+1}} = \frac{k+1}{k-1} \frac{c-k+2}{c-k} > 1,
     \end{equation*}
  this completes the proof.
\end{proof}


\begin{theorem}
\label{log-concave 2}
    The coefficients of Markov polynomial $M_\rho$ that appear on the lines 
   with $j=1$ and $i + j = a + b - 2$ of its Newton polygon
   are strictly log-concave. 
\end{theorem}

\begin{proof}
The proof follows from Theorem \ref{Coefficients} and the previous lemma.
\end{proof}

As for the $3$rd diagonal we are only able to prove the log-concavity in the
case where $\rho\leq3/5$.

\begin{theorem}
\label{logconcave3rd}
    The coefficients of Markov polynomials $M_{\rho}$ that appear on the line with $i + j = a + b - 3$ of its Newton polygon are log-concave for $\frac{a}{b} \leq \frac{3}{5}$.
\end{theorem}

\begin{proof}

We will use the following result quoted by Stanley \cite{Sta} (but which
dates back to Newton \cite{Newton}).

\begin{theorem}[Stanley \cite{Sta}]
\label{thm:stanley}
    Let $p(X) = \sum_{k = 0}^n \alpha_k X^k$. If the polynomial $p(X)$ has only real zeros, then the sequence $\alpha_k$ is log-concave.
\end{theorem}

 Let $\alpha_k$ be the coefficients on the diagonal $i + j - 3$ stated in Theorem~\ref{Coefficients}:
    \begin{equation*}
        \alpha_k = E \binom{a+b-3}{k} + F \binom{a+b-4}{k-1} + G \binom{a+b-5}{k-2}, 
    \end{equation*}
$$
        E = \frac{(a-1)(a-2)}{2},\,\, 
        F = a(b-a-1), \,\,
        G = \frac{1}{2}\left[(b-a)^2 + 5a - 3b \right].
$$
    Let \[p(X) = \sum_{k = 0}^{a+b-3} \alpha_k X^k,\] then we have
    \begin{align*}
        p(X) &= E (1+X)^{a+b-3} + FX (1+X)^{a+b-4} + GX^2 (1+X)^{a+b-5} \\
        &= (1+X)^{a+b-5} \left[E(1+X)^2 + FX(1+X) + GX^2\right]=(1+X)^{a+b-5} q(X),
    \end{align*}
  where
    \begin{align*}
        q(X) &:= (E + F + G) X^2 + (2E + F)X + E \\
        &= \frac{(b-1)(b-2)}{2} X^2 + (ab-4a+2) X + \frac{(a-1)(a-2)}{2}
    \end{align*}
 Since $(1+X)^{a+b-5}$ has clearly real roots, to apply Theorem~\ref{thm:stanley} we need to show that $q(X)$ has only real roots. The discriminant $\Delta$ of $q(X)$ is
    \begin{align*}
        \Delta &= (ab-4a+2)^2 - 4 \frac{(b-1)(b-2)}{2} \frac{(a-1)(a-2)}{2} \\
        &=  (ab-4a+2)^2 - (ab-2a-b+2)(ab-a-2b+2).
    \end{align*}
If $a/b\leq3/5$ then 
$  a+2b \geq 4a + \frac{a}3$ and $2a+b \geq 4a-\frac{a}3,$
so
    \begin{align*}
        \Delta &\geq (ab-4a+2)^2 - \left(ab-4a+2+\frac{a}3\right)
\left(ab-4a+2-\frac{a}3\right) =\frac{a^2}9\geq0.
    \end{align*}
\end{proof}


\section{Continuum limit and entropy function}

Consider the sequence of fractions $\rho_n=a_n/b_n$ tending to some $\alpha \in \mathbb [0,1]$ when $n\to \infty.$ 
Let $(i_n, \, j_n)$ be the sequence of points in the Newton polygons
\beq{Newton_n}
\Delta_{\rho_n}=\{i, j \in \mathbb Z_{\geq 0}: i/a_n+j/b_n\geq 1, \quad
i+j\leq a_n+b_n-1\}
\eeq 
such that the corresponding points $(\xi_n=i_n/b_n, \eta_n=j_n/b_n) \to (\xi, \eta) \in \mathcal N_{\alpha}$ as $n\to \infty,$
where $\mathcal N_{\alpha}$ is the scaled Newton polygon 
\beq{scaledN}
\mathcal N_{\alpha}:=\{\xi,\eta \in \mathbb R: \xi,\eta > 0,\quad 
\xi+\alpha \eta>\alpha, \quad \xi+\eta<\alpha+1.\}
\eeq 

Define the {\it entropy function} of this sequence as 
\beq{entropy}
\mathcal H=\limsup_{n \to \infty} \frac{1}{b_n}\ln A_{i_n,j_n}(\rho_n),
\eeq 
where $A_{i,j}(\rho_n)$ are the coefficients of the Markov polynomial $M_{\rho_n}.$

\begin{proposition}
 The corresponding  limit superior does exist for any such sequences.
\end{proposition}

The proof follows from the estimate $A_{i,j}(\rho_n)<m_{\rho_n},$ where $m_{\rho_n}$ is the corresponding Markov number, and the results of Fock \cite{Fock} (see also \cite{SorVe}).

\medskip
We have two natural questions:

Q.1 {\it Does the entropy $\mathcal H$ depend only on $\xi, \eta$ and $\alpha$, but not on the choice of the sequences?}

We conjecture that at least for irrational $\alpha$ this is the case.

Q.2 {\it What are the analytic properties of the entropy function $\mathcal H_\alpha(\xi,\eta)$ on the corresponding Newton polygon $\mathcal N_\alpha$?}

We conjecture that the entropy function $\mathcal H_\alpha(\xi,\eta)$ is strictly concave on the Newton polygon $\mathcal N_\alpha$ (and, in particular, it is continuous there).

We prove this now for the sequence $\rho_n=1/n$ corresponding to Fibonacci polynomials. In that case the scaled Newton polygon is the triangle with vertices at $(0, 0), (1, 0), (0, 1)$.

Let us start with a well-known fact about binomial coefficients.

Let $m_n$ be a sequence of positive integers such that $m_n/n \to p \in [0,1]$ when $n \to \infty.$

\begin{proposition}[Entropy Function for Binomial] For any such sequence
\label{binom entropy}
    \begin{equation*}
        \lim_{n \to \infty} \frac{1}{n} \ln \binom{n}{m_n} = H(p),
    \end{equation*}
   \beq{H}
    H(p): = -p \ln p - (1-p) \ln(1-p).
    \eeq
\end{proposition}

\begin{proof} Recall the Stirling formula
$$
n!\sim\sqrt{2\pi n}(n/e)^n,
$$
implying that
\beq{Stirling}
\ln n!=n \ln n-n+O(\ln n).
\eeq
Using this we have
$$
\frac{1}{n} \ln \binom{n}{m_n}=\frac{1}{n}\left[-m_n \ln \frac{m_n}{n}-(n-m_n)\ln \frac{n-m_n}{n}+ O(\ln n)\right],
$$
which turns to $H(p)$ when $n$ goes to infinity.
\end{proof}

To compute the entropy of the Fibonacci polynomials $M_{1/n}$ we can use the explicit formula ~\eqref{eqn:Fib Coeff} for the corresponding coefficients $A_{i,j}(1/n).$ 

\begin{proposition}[Entropy Function for Fibonacci Polynomials]
\label{Fib entropy}
    \begin{equation*}
        \lim_{n \to \infty} \frac{1}{n} \ln A_{i_n,j_n}(1/n) = (1-\eta) H\left(\frac{\xi}{1-\eta}\right) + (\xi+\eta) H\left(\frac{\xi}{\xi+\eta}\right),
    \end{equation*}
    where $\xi =   \lim_{n \to \infty}\frac{i_n}{n}, \eta =   \lim_{n \to \infty}\frac{j_n}{n}$ and $H$ is given by (\ref{H}).
\end{proposition}

\begin{proof}
    Substituting $i_n = n \xi_n, j_n = n\eta_n$ into Eq.~\eqref{eqn:Fib Coeff} we have 
    \begin{equation*}
        A_{i_n, j_n} = \binom{n(1-\eta_n)}{n(1-\xi_n-\eta_n)} \binom{n(\xi_n+\eta_n)}{n\eta_n}
    \end{equation*}
and therefore
    \begin{equation*}
        \frac{1}{n} \ln A_{i_n, j_n} = \frac{1}{n} \ln \binom{n(1-\eta_n)}{n(1-\xi_n-\eta_n)} + \frac{1}{n} \ln \binom{n(\xi_n+\eta_n)}{n\eta_n}.
    \end{equation*}
Applying again formula (\ref{Stirling}) and using the fact that $\xi_n\to \xi, \eta_n\to \eta$ we have
    \begin{equation*}
        \lim_{n\to\infty} \frac{1}{n} \ln A_{i_n, j_n} = (1-\eta) H\left(\frac{\xi}{1-\eta}\right) + (\xi+\eta) H\left(\frac{\xi}{\xi+\eta}\right),
    \end{equation*}
    with $H(p) = -p \ln p - (1-p) \ln(1-p)=H(1-p).$
\end{proof}


\begin{theorem}
    The entropy function of Markov-Fibonacci polynomials
    \begin{equation}
    \label{fibent}
       \mathcal{F}(\xi, \eta): = (1- \eta) H\left(\frac{\xi}{1-\eta}\right) + (\xi+ \eta) H\left(\frac{ \xi}{\xi+ \eta}\right)
    \end{equation}
with $H$ given by (\ref{H}), is strictly concave on the region $\xi, \eta > 0, \xi+ \eta < 1$. It is invariant under the change $(\xi,\eta)\to (\xi, 1-\xi-\eta)$.
\end{theorem}

\begin{proof}
  We will prove that the Hessian matrix
     \begin{equation*}
        \Hess(\mathcal{F}) = \begin{pmatrix} \mathcal{F}_{\xi \xi} & 
\mathcal{F}_{\xi \eta} \\
        \mathcal{F}_{\eta \xi } & \mathcal{F}_{\eta \eta} \end{pmatrix}
    \end{equation*}
    is negative definite. By Sylvester's criterion it is enough to show that
       $ \mathcal{F}_{\xi\xi} < 0$ and  $\det \Hess(\mathcal{F}) > 0.$
   Straightforward calculations lead to
    \begin{equation*}
        \Hess(\mathcal{F}) = \begin{pmatrix} \frac{1}{\xi+ \eta-1} - \frac{2}{\xi} + \frac{1}{\xi+ \eta} & \frac{1}{\xi+ \eta-1} + \frac{1}{\xi+ \eta} \\
        \frac{1}{\xi+ \eta-1} + \frac{1}{\xi+ \eta} & \frac{1}{\xi+ \eta-1} + \frac{1}{1- \eta} - \frac{1}{ \eta} + \frac{1}{\xi+ \eta} \end{pmatrix}.
    \end{equation*}
    Clearly $\mathcal{F}_{\xi\xi} < 0$ since $\frac{1}{\xi+ \eta-1} < 0$ and  $\frac{1}{\xi} > \frac{1}{\xi+ \eta} > 0$.
Computing the determinant and simplifying we find
    \begin{equation*}
        \det \Hess(\mathcal{F}) = \frac{1}{ \eta(\xi+ \eta)(1- \eta)(1-\xi- \eta)},
    \end{equation*}
    where all terms in the denominator are positive in the specified region and hence the determinant is positive as required. Therefore the function is strictly concave.
    
    The symmetry $(\xi,\eta)\to (\xi, 1-\xi-\eta)$ is evident from formula (\ref{fibent}).
\end{proof}

The graph of the corresponding function $\mathcal{F}(\xi, \eta)$ is shown in
Figure~\ref{fig:entropy}. 
The maximum value is $2 \ln \frac{1+\sqrt{5}}{2}$ achieved at $\xi=\frac{1}{\sqrt{5}}, \, \eta=\frac{5-\sqrt{5}}{10}.$

\begin{figure}[h]
  \begin{tikzpicture}
    \node[anchor=south west] at (0,0) 
         {\includegraphics[width=95mm]{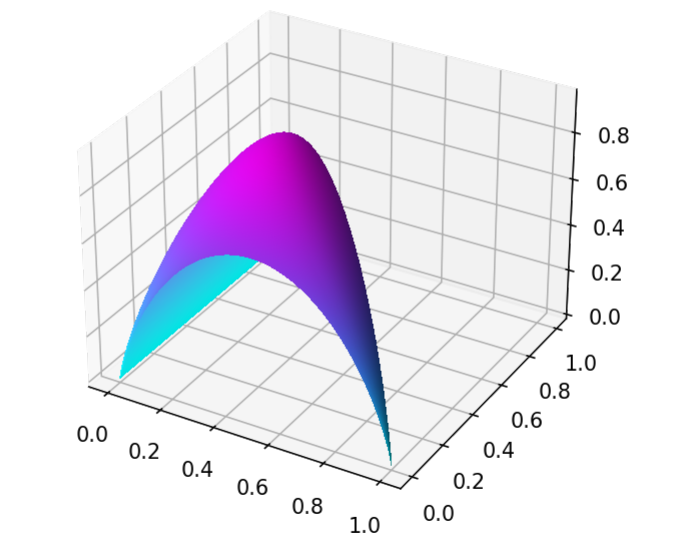}};
    \node at (2.75,0.9) {$\xi$};
    \node at (7.7,1.6) {$\eta$};
  \end{tikzpicture}
\caption{Graph of the entropy for Fibonacci polynomials.}
\label{fig:entropy}
\end{figure} 

\pagebreak
\section{Critical triangle and Markov sails}
\label{sec:Markov Sails}

We can split up the Newton polygon by drawing `critical lines' $i = a, j = b$ as shown in Figure \ref{fig:ctwNP29}. The set of lattice points with
$$
i<a, \; j<b, \; \frac{i}{a}+\frac{j}{b}>1
$$ we will refer to as the {\it critical triangle}. 

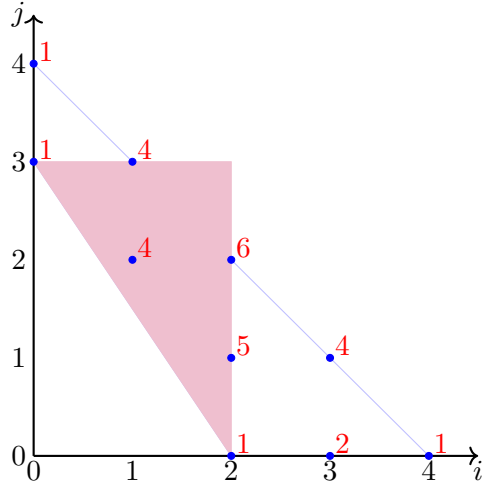
\begin{figure}[H]
    \begin{center}
    \begin{tikzpicture}  [scale=0.65]
            \draw[blue!25] (8, 0) -- (0, 8) -- (0, 6) -- (4, 0) -- (8, 0);
            \filldraw[purple!25] (4, 0) -- (4, 6) -- (0, 6) -- (4, 0);

            \draw[thick, ->] (0, 0) -- (0, 9);
            \draw[thick, ->] (0, 0) -- (9, 0);
            
            \filldraw [blue] (8, 0) circle (2pt);
            \node[red] at (8.25, 0.25) {1};
            \filldraw [blue] (6, 2) circle (2pt);
            \node[red] at (6.25, 2.25) {4};
            \filldraw [blue] (4, 4) circle (2pt);
            \node[red] at (4.25, 4.25) {6};
            \filldraw [blue] (2, 6) circle (2pt);
            \node[red] at (2.25, 6.25) {4};
            \filldraw [blue] (0, 8) circle (2pt);
            \node[red] at (0.25, 8.25) {1};

            \filldraw [blue] (6, 0) circle (2pt);
            \node[red] at (6.25, 0.25) {2};
            \filldraw [blue] (4, 2) circle (2pt);
            \node[red] at (4.25, 2.25) {5};
            \filldraw [blue] (2, 4) circle (2pt);
            \node[red] at (2.25, 4.25) {4};
            \filldraw [blue] (0, 6) circle (2pt);
            \node[red] at (0.25, 6.25) {1};
            
            \filldraw [blue] (4, 0) circle (2pt);
            \node[red] at (4.25, 0.25) {1};

            \node at (8, -0.3) {4};
            \node at (6, -0.3) {3};
            \node at (4, -0.3) {2};
            \node at (2, -0.3) {1};
            \node at (0, -0.3) {0};

            \node at (-0.3, 8) {4};
            \node at (-0.3, 6) {3};
            \node at (-0.3, 4) {2};
            \node at (-0.3, 2) {1};
            \node at (-0.3, 0) {0};

            \node at (-0.3, 9) {$j$};
            \node at (9, -0.3) {$i$};
    \end{tikzpicture}
    \end{center}
    \caption{Critical triangle of the weighted Newton polygon $\Delta_{29}$.
    Notice that per our definition, the lines $i=2$ and $j=3$ themselves are not
    considered as part of the critical triangle.} 
\label{fig:ctwNP29}
\end{figure}

In this section we will look at Markov coefficients inside the critical triangle in more detail. We begin with the following conjecture based on the analysis of experimental data.

\begin{conjecture} (Factor 4 Conjecture)
The coefficients of Markov polynomials strictly
within the critical triangle are divisible by 4.
\end{conjecture}



We are going to be more specific about the coefficients on the lower boundary of the critical triangle, using the following geometric interpretation of continued fractions proposed by Felix Klein.

 Take the region $[0, a] \times [0, b]$ for some positive integers $a, b$ and divide it into two parts with the line $y = \frac{b}{a} x$. Then consider the convex hull of integer points in each of the two parts, the boundary either side of the line $y = \frac{b}{a} x$ will be a broken line that is called, following Arnold, the {\it sail} (see \cite{Arnold,Kar}). More precisely,
\begin{itemize}
    \item The vertices on the sail containing the point $(1, 0)$ are defined to be $A_0, A_1, \dots, A_n$ (where $A_0 = (1, 0)$). This is the lower sail (i.e., below the line).
    \item The vertices on the sail containing the point $(0, 1)$ are defined to be $B_0, B_1, \dots, B_n$ (where $B_0 = (0, 1)$). This is the upper sail (i.e., above the line).
\end{itemize}
One can prove that these vertices on the sails can be found as the convergents of the continued fraction for $\frac{b}{a}$ (see for example \cite{Kar}).

In particular, we have the following procedure to find the Klein diagram. Plot points $(q_k, p_k)$ to represent the continuants $C_k$ alongside the line $y = \frac{b}{a} x$, then connect with straight lines the the points $(q_k, p_k)$ and $(q_{k+2}, p_{k+2})$. Note that the points representing even-indexed continuants will all lie above the line $y = \frac{b}{a} x$ whilst the odd-indexed ones will lie below.

\begin{example}
\label{53sail}
    Consider the rational $\frac{5}{3} = [1, 1, 2]$, its continuants are given in the following table.
    \begin{center}
    \begin{tabular}{|c|c|c|c|c|c|}
        \hline
        & & & 1 & 1 & 2 \\
        \hline
        $p_k$ & 0 & 1 & 1 & 2 & 5 \\
        $q_k$ & 1 & 0 & 1 & 1 & 3 \\
        \hline
    \end{tabular}
    \end{center}
    ($C_{-1} = \frac{0}{1}, C_0 = \frac{1}{0}, C_1 = \frac{1}{1}, C_2 = \frac{2}{1}, C_3 = \frac{5}{3}$).
    
    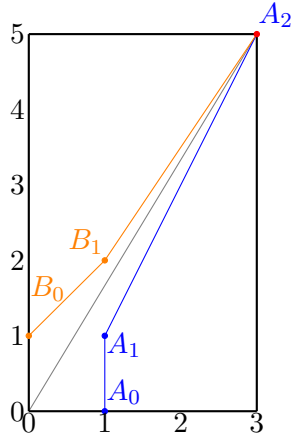
\begin{figure}[H]
    \begin{center}
    \begin{tikzpicture}  [scale=0.5]
            \draw[black!50] (0, 0) -- (6, 10);
    
            \draw[thick] (0, 0) -- (0, 10);
            \draw[thick] (0, 0) -- (6, 0);
            \draw[thick] (6, 0) -- (6, 10);
            \draw[thick] (0, 10) -- (6, 10);
            
            \filldraw[blue] (2, 0) circle (2pt);
            \filldraw[blue] (2, 2) circle (2pt);
            \draw[blue] (2, 0) -- (2, 2);
            \draw[blue] (2, 2) -- (6, 10);
            
            \filldraw[orange] (0, 2) circle (2pt);
            \filldraw[orange] (2, 4) circle (2pt);
            \draw[orange] (0, 2) -- (2, 4);
            \draw[orange] (2, 4) -- (6, 10);
            
            \filldraw[red] (6, 10) circle (2pt);

            \node[blue] at (2.5, 0.5) {$A_0$};
            \node[blue] at (2.5, 1.75) {$A_1$};
            \node[blue] at (6.5, 10.5) {$A_2$};

            \node[orange] at (0.5, 3.25) {$B_0$};
            \node[orange] at (1.5, 4.5) {$B_1$};
    
            \node at (6, -0.3) {3};
            \node at (4, -0.3) {2};
            \node at (2, -0.3) {1};
            \node at (0, -0.3) {0};
    
            \node at (-0.3, 10) {5};
            \node at (-0.3, 8) {4};
            \node at (-0.3, 6) {3};
            \node at (-0.3, 4) {2};
            \node at (-0.3, 2) {1};
            \node at (-0.3, 0) {0};
    \end{tikzpicture}
    \end{center}
    \caption{Klein Diagram for the rational $\frac{5}{3}$}
    \label{fig:Klein53}
    \end{figure}
\end{example}

More formally we can define $A_i$'s and $B_i$'s as follows,
\begin{equation}
\label{sail points}
    A_i = (q_{2i-1}, p_{2i-1}), \qquad B_i = (q_{2i}, p_{2i}),
\end{equation}
where $\frac{p_k}{q_k} = [a_1, a_2, \dots, a_k]$ is the $k$th convergent. Note that for rationals we always consider the final continuant (i.e., the point $(a, b)$) to be part of both sails.

There may be more integer points lying on the sail, indeed in the Example~\ref{53sail} the point $(2, 3)$ lies on the line between $A_1$ and $A_2$. These additional integer points can actually be found immediately if we instead look at the convergents of the Hirzebruch (negative) continued fraction instead, as it was shown by Ustinov in \cite{Ust}.

From the theory of continued fractions, we know
\begin{equation*}
    p_k = a_{k} p_{k-1} + p_{k-2}, \qquad q_k = a_{k} q_{k-1} + p_{k-2}.
\end{equation*}
This means that the line joining the second-last convergent's coordinates $(p_{k-2}, q_{k-2})$ to the corner will have gradient equal to the penultimate convergent. This means that if we flip one of the sails, so that they are both on the same side of the diagonal, they will intersect at the point corresponding to the final convergent. In this way we can combine the sails onto one side. More explicitly, if we consider the coordinate system of the lower sail (blue) to have origin at $(a, b)$ and directions reversed from normal then we can connect the sails. In practise, continuing on from Example~\ref{53sail} this is shown 
on the left in Figure~\ref{fig:Combined53}.

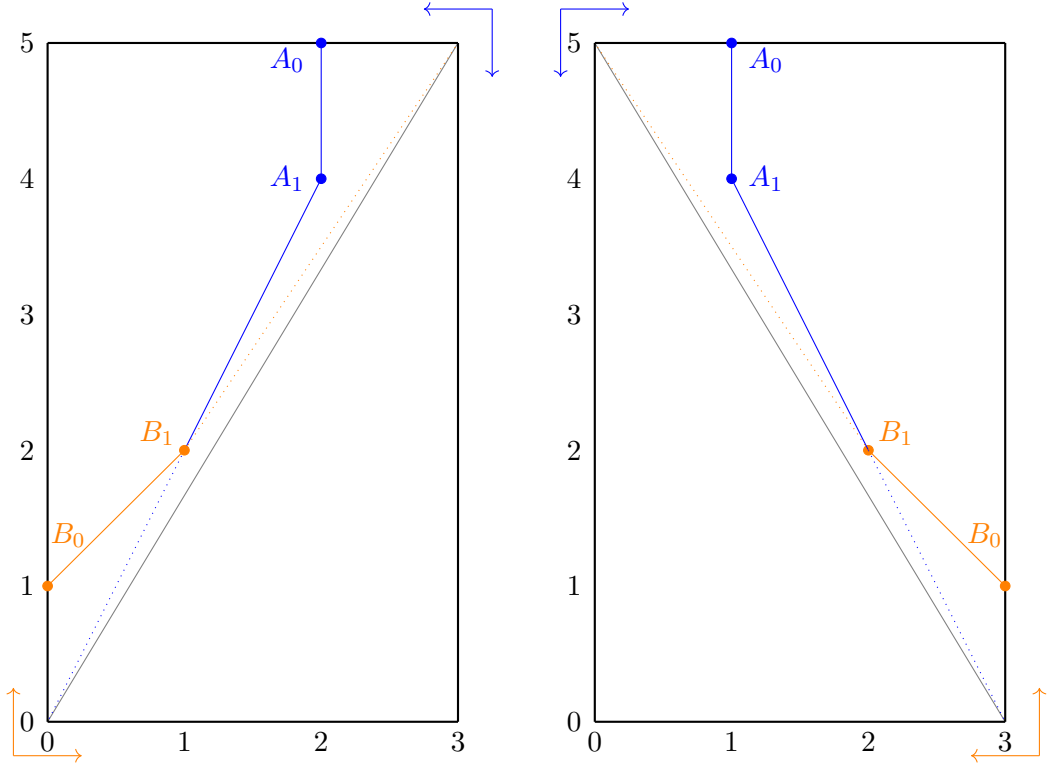
\begin{figure}[H]
    \begin{center}
     \begin{tikzpicture}  [scale=0.9]
            \draw[black!50] (0, 0) -- (6, 10);
    
            \draw[thick] (0, 0) -- (0, 10);
            \draw[thick] (0, 0) -- (6, 0);
            \draw[thick] (6, 0) -- (6, 10);
            \draw[thick] (0, 10) -- (6, 10);

            \draw[blue, ->] (6.5, 10.5) -- (6.5, 9.5);
            \draw[blue, ->] (6.5, 10.5) -- (5.5, 10.5);
            \draw[orange, ->] (-0.5, -0.5) -- (0.5, -0.5);
            \draw[orange, ->] (-0.5, -0.5) -- (-0.5, 0.5);
            
            \filldraw[blue] (4, 10) circle (2pt);
            \filldraw[blue] (4, 8) circle (2pt);
            \draw[blue] (4, 10) -- (4, 8);
            \draw[blue] (4, 8) -- (2, 4);
            \draw[blue, dotted] (2, 4) -- (0, 0);
            
            \filldraw[orange] (0, 2) circle (2pt);
            \filldraw[orange] (2, 4) circle (2pt);
            \draw[orange] (0, 2) -- (2, 4);
            \draw[orange, dotted] (2, 4) -- (6, 10);

            \node[blue] at (3.5, 9.75) {$A_0$};
            \node[blue] at (3.5, 8) {$A_1$};

            \node[orange] at (0.3, 2.75) {$B_0$};
            \node[orange] at (1.6, 4.25) {$B_1$};
    
            \node at (6, -0.3) {3};
            \node at (4, -0.3) {2};
            \node at (2, -0.3) {1};
            \node at (0, -0.3) {0};
    
            \node at (-0.3, 10) {5};
            \node at (-0.3, 8) {4};
            \node at (-0.3, 6) {3};
            \node at (-0.3, 4) {2};
            \node at (-0.3, 2) {1};
            \node at (-0.3, 0) {0};
            \begin{scope}[xshift=8cm]
            \draw[black!50] (6, 0) -- (0, 10);

            \draw[thick] (0, 0) -- (0, 10);
            \draw[thick] (0, 0) -- (6, 0);
            \draw[thick] (6, 0) -- (6, 10);
            \draw[thick] (0, 10) -- (6, 10);

            \draw[blue, ->] (-0.5, 10.5) -- (-0.5, 9.5);
            \draw[blue, ->] (-0.5, 10.5) -- (0.5, 10.5);
            \draw[orange, ->] (6.5, -0.5) -- (5.5, -0.5);
            \draw[orange, ->] (6.5, -0.5) -- (6.5, 0.5);
            
            \filldraw[orange] (6, 2) circle (2pt);
            \node[orange] at (5.7, 2.75) {$B_0$};
            \filldraw[orange] (4, 4) circle (2pt);
            \node[orange] at (4.4, 4.25) {$B_1$};
            \filldraw[blue] (2, 8) circle (2pt);
            \node[blue] at (2.5, 8) {$A_1$};
            \filldraw[blue] (2, 10) circle (2pt);
            \node[blue] at (2.5, 9.75) {$A_0$};

            \draw[blue] (2, 10) -- (2, 8);
            \draw[blue] (2, 8) -- (4, 4);
            \draw[orange] (6, 2) -- (4, 4);            

            \draw[blue, dotted] (4, 4) -- (6, 0);
            \draw[orange, dotted] (4, 4) -- (0, 10);

            \node at (6, -0.3) {3};
            \node at (4, -0.3) {2};
            \node at (2, -0.3) {1};
            \node at (0, -0.3) {0};

            \node at (-0.3, 10) {5};
            \node at (-0.3, 8) {4};
            \node at (-0.3, 6) {3};
            \node at (-0.3, 4) {2};
            \node at (-0.3, 2) {1};
            \node at (-0.3, 0) {0};
            \end{scope}
    \end{tikzpicture}
    \end{center}
    \caption{(Left) Combined sail for the rational $\frac{5}{3}$;
(Right) the reflection in the line $x=a/2$.}
    \label{fig:Combined53}
\end{figure}

This is precisely the convex hull of the integer lattice above the line $y = \frac{b}{a} x$ (with endpoints $(0, 0), (3, 5)$ omitted).

From Eq.~\eqref{eqn:Newton Poly} it is clear to see that the critical triangle of the Newton polygon consists of a region $[0, a] \times [0, b]$ split \textit{via} a central diagonal, which is comparable to that of the Klein diagram. Furthermore, we are again only interested in integral points.

Indeed, reflecting the leftmost picture in Figure~\ref{fig:Combined53} in the line $x = \frac{a}{2}$ we obtain the critical triangle, where coordinate frames have been shown in their corresponding colours.

The integer points on this combined sail all refer to a monomial in the Markov polynomial and as such they each have their own coefficient. We will refer to this combined sail, together with the coefficients at each point as the \textit{Markov sail}.

With our change of coordinate frames, we reformulate the definition of points from Eq.~\eqref{sail points} to
\begin{equation*}
\label{scaled sail points}
    \widetilde{A_i} = (q_{2i-1}, b-p_{2i-1}), \qquad \widetilde{B_i} = (a - q_{2i}, p_{2i}),
\end{equation*}
but going forward we will omit the tilde's for simplicity.

\subsection{Markov Sail Duality}
There is an important {\it Edge-Angle duality} between the two sails, saying that the length of a unbroken line on one sail is equal to the `index' at a vertex/rational angle on the opposite sail, see Karpenkov \cite{Kar}. 

Here the {\it integer length} (denoted $\text{l}\ell(AB)$) of a line segment $AB$ is defined to be the number of integer points on the its interior plus $1$, and the {\it index} (denoted $\text{l}\alpha(\angle{BAC})$) of a rational angle $\angle{BAC}$ is defined to be the index of the sublattice generated by the integer vectors of the lines $AB$ and $AC$ in the integer lattice, see \cite{Kar}.

Formally, Karpenkov stated this duality as follows
\begin{align*}
    \text{l}\alpha(\angle{A_i A_{i+1} A_{i+2}}) = \text{l}\ell(B_i B_{i+1}) &= a_{2i+2}, \\ 
    \text{l}\alpha(\angle{B_i B_{i+1} B_{i+2}}) = \text{l}\ell(A_{i+1} A_{i+2}) &= a_{2i+3},
\end{align*}
where $a_i$'s are the partial quotients of the continued fraction $\frac{b}{a} = [a_1, a_2, \dots]$.

We conjecture that in our case (for the so called Markov sails) we have a sort of duality between the `weights' of these vertices/lines on the sail. Denote $\mathcal{M}(C)$ to be the coefficient of the monomial represented at the point $C$ on the Newton polygon, we will refer to these as \textit{$\mathcal{M}$-values} and define this only for integer points in the interior of the critical triangle (i.e., not on the boundary).

It appears as though, on unbroken lines of the sail, the coefficients differ according to an arithmetic progression. Denote $d(C_i C_{i+1})$ to be the common difference of $\mathcal{M}$-values of integer points along an unbroken line segment of the Markov sail between vertices $C_i$ and $C_{i+1}$. We conjecture the following extended duality.

\begin{conjecture} (Markov Sail Duality) The coefficients of Markov polynomials satisfy Markov sail duality
\label{extended duality}
    \begin{align*}
        d(B_i B_{i+1}) &= -\mathcal{M}(A_{i+1}) \\
        d(A_{i+1} A_{i+2}) &= -\mathcal{M}(B_{i+1})
    \end{align*}
\end{conjecture}

This conjecture will provide us with a method for determining almost all $\mathcal{M}$-values on the sail if we having some starting value of which to work from. This is because we known how many integer points lie on each unbroken line from the duality, Corollary 3.1 in \cite{Kar}, so can compute $\mathcal{M}$-values using the equations from this conjecture from some initial condition.

With regard to this initial starting value, we claim that the $\mathcal{M}$-value corresponding to the point of the penultimate continuant is always $4$. 

\begin{conjecture}(Location of 4)
\label{four}
    Consider the continued fraction $\frac{b}{a} = [a_1, a_2, \dots, a_n]$. If $n = 2m+1$ (odd) then $\mathcal{M}(B_m) = 4$. If $n = 2m$ (even) then $\mathcal{M}(A_m) = 4$.
\end{conjecture}

From this value we would then be able to work backwards (alternating between each side of the sail) to recover the $\mathcal{M}$-values on most of the sail (all but on line $A_0 A_1$ if there are any integer points on its interior). To see this, consider the following example.

\begin{example}
\label{ex:18/13}
$\frac{b}{a} = \frac{18}{13} = [1, 2, 1, 1, 2]$
\begin{center}
    \begin{tabular}{|c|c|c|c|c|c|c|c|}
        \hline
        & & & 1 & 2 & 1 & 1 & 2 \\
        \hline
        $p_k$ & 0 & 1 & 1 & 3 & 4 & 7 & 18 \\
        $q_k$ & 1 & 0 & 1 & 2 & 3 & 5 & 13 \\
        \hline
    \end{tabular}
    \end{center}
According to Eq.~\eqref{scaled sail points} we therefore have $A_0 = (1, 18), A_1 = (1, 17), A_2 = (3, 14)$ and $B_0 = (13, 1), B_1 = (11, 3), B_2 = (8, 7)$.

\begin{figure}[H]
    \begin{center}
    \begin{tikzpicture}[scale=0.6]
            \draw[black!50] (13, 0) -- (0, 18);

            \draw[thick] (0, 0) -- (0, 18);
            \draw[thick] (0, 0) -- (13, 0);
            \draw[thick] (13, 0) -- (13, 18);
            \draw[thick] (0, 18) -- (13, 18);

            \draw[blue, ->] (-1.5, 18.5) -- (-1.5, 17.5);
            \draw[blue, ->] (-1.5, 18.5) -- (-0.5, 18.5);
            \draw[orange, ->] (14.5, -0.5) -- (13.5, -0.5);
            \draw[orange, ->] (14.5, -0.5) -- (14.5, 0.5);
            
            \filldraw[orange] (13, 1) circle (2pt);
            \node[orange] at (13.5, 1) {$B_0$};
            \filldraw[orange] (11, 3) circle (2pt);
            \node[orange] at (11.5, 3.5) {$B_1$};
            \filldraw[orange] (8, 7) circle (2pt);
            \node[orange] at (8.5, 7.5) {$B_2$};
            
            \filldraw[blue] (1, 18) circle (2pt);
            \node[blue] at (1, 18.5) {$A_0$};
            \filldraw [blue] (1, 17) circle (2pt);
            \node[blue] at (1.75, 17) {$A_1$};
            \filldraw [blue] (3, 14) circle (2pt);
            \node[blue] at (3.5, 14.5) {$A_2$};

            \draw[orange] (13, 1) -- (11, 3);
            \draw[orange] (11, 3) -- (8, 7);
            \draw[blue] (1, 18) -- (1, 17);
            \draw[blue] (1, 17) -- (3, 14);
            \draw[blue] (3, 14) -- (8, 7);

            \node at (13, -0.3) {13};
            \node at (12, -0.3) {12};
            \node at (11, -0.3) {11};
            \node at (8, -0.3) {8};
            \node at (3, -0.3) {3};
            \node at (1, -0.3) {1};
            \node at (0, -0.3) {0};

            \node at (-0.3, 18) {18};
            \node at (-0.3, 17) {17};
            \node at (-0.3, 14) {14};
            \node at (-0.3, 7) {7};
            \node at (-0.3, 3) {3};
            \node at (-0.3, 2) {2};
            \node at (-0.3, 1) {1};
            \node at (-0.3, 0) {0};
    \end{tikzpicture}
    \caption{Markov sail for Example~\ref{ex:18/13}.}
    \end{center}
\end{figure}
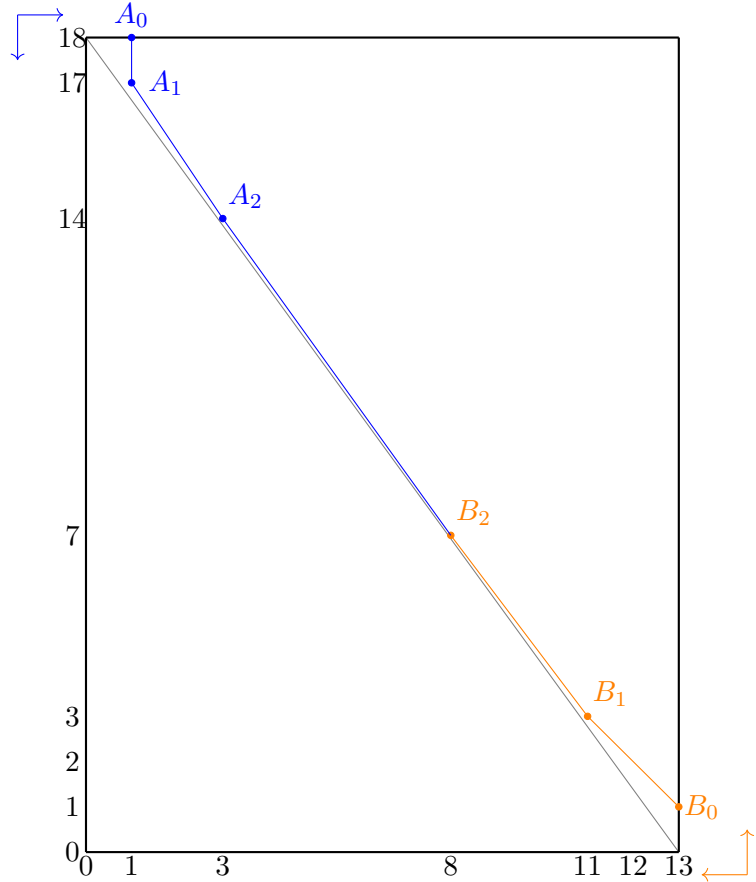
\end{example}

According to Conjecture~\ref{four}, since the continued fraction has odd length ($n = 5, m = 2$) we should have $ m(B_2) = 4$. Now using Conjecture~\ref{extended duality}, we should have
\begin{align*}
    d(A_2 A_3) &= -\mathcal{M}(B_2) \\
    d(B_1 B_2) &= -\mathcal{M}(A_2) \\
    d(A_1 A_2) &= -\mathcal{M}(B_1) \\
    d(B_0 B_1) &= -\mathcal{M}(A_1).
\end{align*}
Note the integer lengths
\begin{align*}
    \text{l}\ell(A_2 A_3) &= a_5 = 2 \\
    \text{l}\ell(B_1 B_2) &= a_4 = 1, \\
    \text{l}\ell(A_1 A_2) &= a_3 = 1, \\
    \text{l}\ell(B_0 B_1) &= a_2 = 2.
\end{align*}
Noting that $B_m$ is the penultimate integer point before $A_{m+1}$ on the line $A_m A_{m+1}$ for $n$ odd (this will always be the case), we have
\begin{align*}
    \mathcal{M}(A_2) &= \mathcal{M}(B_2) + (a_5 - 1) \mathcal{M}(B_2) = 8, \\
    \mathcal{M}(B_1) &= \mathcal{M}(B_2) + a_4 \mathcal{M}(A_2) = 12, \\
    \mathcal{M}(A_1) &= \mathcal{M}(A_2) + a_3 \mathcal{M}(B_1) = 20.
\end{align*}
Since $B_0$ is not in the interior of the critical triangle we cannot apply to conjecture to find its $\mathcal{M}$-value. However, since the integer length of $B_0 B_1$ is $2$ there will be an integer point on the interior of this line (which will lie in the interior of the critical triangle). This interior point however can be found from the conjecture and should have $\mathcal{M}$-value equal to $$\mathcal{M}(B_1) + \mathcal{M}(A_1) = 32.$$

Checking numerically, all of the $\mathcal{M}$-values calculated in this example agree with the actual $\mathcal{M}$-values, supporting the conjectures in this case. 




\section{Sails of Pell polynomials}

Returning to our `simplest' cases of Markov polynomials, we observe that the critical triangle (and hence sail) is empty for Fibonacci Markov polynomials (as $a = 1$). However, in the case of Markov polynomials corresponding to Pell numbers we can prove some of our earlier conjectures using Eq.~\eqref{Pell Coeff Rec}.

The first of which is the presence and location of the coefficient $4$. According to our conjecture the coefficient $4$ should lie at coordinates corresponding to the penultimate convergent of $\frac{b}{a}$. Here
\begin{equation*}
    \frac{b}{a} = \frac{n+1}{n} = [1, n]
\end{equation*}
So the penultimate convergent is $[1] = \frac{1}{1}$ and we expect the location of the $4$ to be at $(1, n)$ ($1$ in each direction from corner $(0, n+1)$ since continued fraction is even length).

\begin{theorem}
\label{Pell 4 Location}
    For Markov polynomials $M_{\rho}$, where $\rho = \frac{n}{n+1}$, the coefficient at the coordinate $(1, n)$ in the Newton polygon is $4$.
\end{theorem}

\begin{proof}
Looking at $i = 1, j = n$ in Eq.~\eqref{Pell Coeff Rec}, we find
\begin{equation}
    A_{1, n}^{(2n+1)} = \left[A_{-1, n}^{(2n-1)} + 2 A_{0, n-1}^{(2n-1)} + A_{1, n-2}^{(2n-1)}\right] + \left[A_{0, n}^{(2n-1)} + A_{1, n-1}^{(2n-1)}\right] - A_{0, n-1}^{(2n-3)}.
\end{equation}
Clearly if $i$ or $j$ is negative the coefficient will be $0$ however this is not the only case where the coefficient is $0$. The lower vertices of the Newton polygon for $R_{2n-1}$ are at $(0, n)$ and $(n-1, 0)$ therefore $A_{0, n-1}^{(2n-1)}, A_{1, n-2}^{(2n-1)}$ will also be $0$ (they are outside the boundaries of the Newton polygon - can be seen rigorously by showing $\frac{i}{n-1} + \frac{j}{n} \leq 1$).
This leaves the final $3$ terms; $A_{0, n}^{(2n-1)}$ and $A_{0, n-1}^{(2n-3)}$ are vertices of their respective Newton polygons and are known to be $1$ (and cancel each other out) so we are left with
    $A_{1, n}^{(2n+1)} = A_{1, n-1}^{(2n-1)}.$
So all we need to do is prove that it holds in the base case ($n = 2$, since critical triangle is empty in $n = 1$ case). Direct computation of $M_{2/3}$ (corresponding to the Markov number  $29$) shows that this is indeed the case (see Figure~\ref{fig:ctwNP29}).
\end{proof}

Combining our conjectures regarding coefficients on the sail, in the case $M_{n/n+1}$ the integer points on the sail should have coefficients following an arithmetic progression with common difference $4$.

\begin{theorem}
    For Markov polynomials $M_{\rho}$ where $\rho = \frac{n}{n+1}$, the coefficients at the coordinates $(m, n+1-m)$ in the Newton polygon are $4m$, for $m = 1, 2, \dots, n-1$.
\end{theorem}

\begin{proof}
    For $m = 1$ the result is proven - it is precisely Theorem \ref{Pell 4 Location}.

    To complete the proof we need to consider two seperate cases; $m = 2, \dots, n-2$ and $m = n - 1$.
    In both cases, Eq.~\eqref{Pell Coeff Rec} becomes
    \begin{multline*}
        A_{m, n+1-m}^{(2k+1)} = \left[A_{m-2, n+1-m}^{(2k-1)} + 2 A_{m-1, n-m}^{(2k-1)} + A_{m, n-m-1}^{(2k-1)}\right] \\
        + \left[A_{m-1, n+1-m}^{(2k-1)} + A_{m, n-m}^{(2k-1)}\right] - A_{m-1, n-m}^{(2k-3)}.
    \end{multline*}
    For $m = 1, 2, \dots, n-1$, the first $3$ terms on the RHS are $0$ (outside the Newton polygon), leaving
    \begin{equation*}
        A_{m, n+1-m}^{(2k+1)} = \left[A_{m-1, n+1-m}^{(2k-1)} + A_{m, n-m}^{(2k-1)}\right] - A_{m-1, n-m}^{(2k-3)}.
    \end{equation*}
    By induction this becomes
$
        A_{m, n+1-m}^{(2k+1)} = \left[4(m-1) + 4m\right] - 4(m-1)= 4m,
$
    where we can check the first two (relevant) cases directly $k = 2, 3$.

    For $m = n - 1$, the term $A_{m, n-m-1}^{(2k-1)} = A_{n-1, 0}^{(2k-1)}$ does not disappear, in fact it corresponds to a lower boundary point and equals $1$ and also the coefficients $A_{m, n-m}^{(2k-1)} = A_{n-1, 1}^{(2k-1)}$ and $A_{m-1, n-m}^{(2k-3)} = A_{n-2, 1}^{(2k-3)}$ lie on the boundary of the critical triangle (not in the interior so we can use the same inductive argument).

    However, Theorem \ref{2nd Hor} tells us that the first point on the first horizontal $(j = 1)$ has coefficient $A_{n-1, 1}^{(2k-1)} = 3(n-1) - 1$ and $A_{n-2, 1}^{(2k-3)} = 3(n-2) - 1$ and so Eq.~\eqref{Pell Coeff Rec} becomes
    \begin{align*}
        A_{n-1, 2}^{(2k+1)} &= A_{n-1, 0}^{(2k-1)} + \left[A_{n-2, 2}^{(2k-1)} + A_{n-1, 1}^{(2k-1)}\right] - A_{n-2, 1}^{(2k-3)} \\
        &= 1 + \left[4(m-1) + (3(n-1)-1)\right] - (3(n-2) - 1)=4m
    \end{align*}
    as required.
\end{proof}

So now we have proven the coefficients on the whole of the interior of the sail. As mentioned we also know one of the boundary coefficients of the sail, namely $A_{n, 1}^{(2k+1)} = 3n - 1$. If we can determine the other boundary $A_{1, n+1}^{(2k+1)}$ then we will know the sail coefficients in their entirety.

Again using Eq.~\eqref{Pell Coeff Rec},
\begin{equation*}
    A_{1, n+1}^{(2k+1)} = \left[A_{-1, n+1}^{(2k-1)} + 2 A_{0, n}^{(2k-1)} + A_{1, n-1}^{(2k-1)}\right] + \left[A_{0, n+1}^{(2k-1)} + A_{1, n}^{(2k-1)}\right] - A_{0, n}^{(2k-3)}.
\end{equation*}
Since $$A_{-1, n+1}^{(2k-1)}=0,\, A_{0, n}^{(2k-1)}=1,\, A_{1, n-1}^{(2k-1)}=4,\, A_{0, n+1}^{(2k-1)}=n-2, \, A_{0, n}^{(2k-3)}=n-3,$$
we have 
$$
    A_{1, n+1}^{(2k+1)} = \left[0 + 2(1) + 4\right] + \left[n-2 + A_{1, n}^{(2k-1)}\right] - (n-3) 
   = A_{1, n}^{(2k-1)} + 7.
$$
Computing $A_{1, 3}^{(3)} = 4$, the solution to this recursion is
$
    A_{1, n+1}^{(2k+1)} = 7n - 10, 
$
so now we have the weights on the sail in its entirety: from top to bottom 
    $(7n - 10,\ 4,\ 8,\ \dots,\ 4n-4,\ 3n-1).$

\section{Concluding remarks}

As we have seen, there are many questions about Markov polynomials, which are still to be answered. Already the question if they provide (up to change of sign of any two of them) all integer Laurent-polynomial solutions to the generalised Markov equation (\ref{MEq}) is still open.
This can be considered as a Laurent-polynomial analogue of the strong approximation conjecture by Bourgain, Gamburd and Sarnak \cite{BGS} about connectivity of the Markov graph of the $\mathrm{mod}$-$p$ solutions of the Markov equation (see recent progress in this direction in \cite{EFLMT}).

Note that from Theorem~\ref{denom} it follows that all Markov
polynomials are distinct, which implies that for a generic choice of
the hyperbolic structure on a once-punctured torus, no two simple
closed geodesics have the same length \cite{Pro}. On the relation with
hyperbolic geometry and the Teichm\"uller theory we refer to
\cite{FST, Fock, Penner}.

We have shown also that there are very plausible conjectures about the coefficients of Markov polynomials.
They can be reformulated combinatorially using the weighted perfect matching interpretation of these coefficients proposed in \cite{IMPV,Pro,CS}.
It would be interesting to explore if this combinatorial interpretation can help to prove some of them (in particular, the Factor-4 conjecture).

\section*{Acknowledgements.}
We would like to thank Jonah Gaster, Oleg Karpenkov, Dylan Thurston, Li Li and especially Ralf Schiffler and Alexey Ustinov for very helpful
discussions.

One of us (Sam Evans) is grateful to the Higher Education, Research and Innovation Department of the French Embassy in the UK for the support of his visit to Laboratoire de Math\'ematiques de Reims UMR 9008 in June 2025.


\end{document}